\documentclass[a4paper,final,10pt]{article}
\usepackage[utf8]{inputenc}
\usepackage{subfigure}
\usepackage[colorinlistoftodos]{todonotes}
\usepackage{xcolor,pict2e}
\usepackage{amsmath}
\usepackage{amssymb}
\usepackage{mathabx}
\usepackage{verbatim}
\usepackage{bm}
\usepackage{ulem}
\usepackage{cleveref}

\newcommand*{\rom}[1]{\expandafter\@slowromancap\romannumeral #1@}

\newtheorem{theorem}{Theorem}[section]

\newtheorem{lemma}[theorem]{Lemma}

\newtheorem{remark}[theorem]{Remark}

\newenvironment{proof}{\textit{Proof:}}{$\hfill\square$\newline}

\newcommand\apc[1]{\textcolor{purple}{#1}}

\newcommand\din{\Delta^{\text{in}}}
\newcommand\dout{\Delta^{\text{out}}}
\newcommand\pin{\mathbf{p}^{\text{in}}}
\newcommand\redout{\bgroup\markoverwith{\textcolor{red}{\rule[.5ex]{2pt}{0.8pt}}}\ULon}

\title{Slow passage through a transcritical bifurcation \\
in piecewise linear differential systems: \\
canard explosion and enhanced delay}

\author{A. Pérez-Cervera\footnotemark[1]
\and A. E. Teruel\footnotemark[3]}

\begin{document}

\maketitle
\renewcommand{\thefootnote}{\fnsymbol{footnote}}
\footnotetext{Keywords: Piecewise linear systems, dynamic bifurcations, slow passage, transcritic bifurcation, enhanced delay.}
\footnotetext[1]
{Serra Húnter Fellow. Departament de Matem\`atiques. Universitat Politècnica de Catalunya, Barcelona, Spain.}
\footnotetext[2]
{Departament de Matem\`atiques i Inform\`atica \& IAC3, Universitat de les Illes Balears, Palma de Mallorca, Spain.}

\begin{abstract}

In this paper we analyse the phenomenon of the slow passage through a transcritical bifurcation with special emphasis in the maximal delay $z_d(\lambda,\varepsilon)$ as a function of the bifurcation parameter $\lambda$ and the singular parameter $\varepsilon$. We  quantify the maximal delay by constructing a piecewise linear (PWL) transcritical minimal model and studying the dynamics near the slow-manifolds. Our findings encompass all potential maximum delay behaviours within the range of parameters, allowing us to identify: i) the trivial scenario where the maximal delay tends to zero with the singular parameter; ii) the singular scenario where $z_d(\lambda, \varepsilon)$ is not bounded, and also iii) the transitional scenario where the maximal delay tends to a positive finite value as the singular parameter goes to zero. Moreover, building upon the concepts by A.~Vidal and J.P.~Françoise (Int.~J.~Bifurc. Chaos~Appl. 2012), we construct a PWL system combining symmetrically two transcritical minimal models in such a way it shows periodic behaviour. As the parameter $\lambda$ changes, the system presents a non-bounded canard explosion leading to an enhanced delay phenomenon at the critical value. Our understanding of the maximal delay $z_d(\lambda, \varepsilon)$ of a single normal form, allows us to determine both, the amplitude of the canard cycles and, in the enhanced delay case, the increase of the amplitude for each passage.

\end{abstract}

\section{Introduction}

Slow passage or delayed bifurcations are dynamical phenomena that typically occur in ordinary differential equations with slow drifting parameters. In particular, consider a one-parameter family of differential equations $\dot{\mathbf{x}}=\mathbf{f}(\mathbf{x},y)$ and assume slow dynamics for parameter $y$, that is $y=y_0+\varepsilon t$ with $0<\varepsilon\ll 1$, what can also be expressed as the differential equation 
\begin{equation}\label{def:delayedbif}
\begin{array}{l}
\dot{\mathbf{x}}=\mathbf{f}(\mathbf{x},y),\\ 
\dot{y}=\varepsilon.
\end{array}
\end{equation}
Due to the timescale separation, the differential system \eqref{def:delayedbif} is commonly referred to as a slow-fast system, whereby $\mathbf{x}$  represents the fast variable and $y$ represents the slow variable. For such slow-fast systems, one aims to describe the flow of the full system \eqref{def:delayedbif} for $\varepsilon>0$ starting from the flow of the fast subsystem, that is, System \eqref{def:delayedbif} with $\varepsilon=0$.


Following Fenichel's results \cite{F79,J95}, compact normally hyperbolic invariant manifolds of the fast subsystem persist as slow manifolds for $\varepsilon>0$, and they are located at a $O(\varepsilon)$ Hausdorff distance of the original manifold. Fenichel's theory also describes the behaviour of the flow along the slow manifold and surrounding it.  

In particular, equilibrium points of the differential equation  $\dot{\mathbf{x}}=\mathbf{f}(\mathbf{x},y)$ give rise to a manifold of equilibria of the fast subsystem \eqref{def:delayedbif} called the critical manifold, $\mathcal{S}_0=\{(\mathbf{x},y):\mathbf{f}(\mathbf{x},y)=0\}$.  While the equilibrium point of the equation is hyperbolic, the critical manifold is normally hyperbolic and persists as a slow manifold $\mathcal{S}_\varepsilon$ for $\varepsilon>0$. Therefore, the behaviour of the flow of the full system \eqref{def:delayedbif}, on and around the slow manifold $\mathcal{S}_\varepsilon$, is described by the Fenichel's results. However, when the equilibrium point loses its hyperbolicity at a bifurcation by changing stability, branches of the slow manifold with different stability properties usually appear. In this scenario, the evolution of the slow manifolds and of the flow around them is not predicted by Fenichel's theory and depends on the interplay between the slow manifolds. Under suitable conditions, there exist trajectories that evolve close to the repelling slow manifold for a time after the bifurcation, which is regarded as a delay in the loss of stability. Therefore, canard trajectories \cite{B81,DDR21} appear as a transitional dynamical objects in the loss of stability phenomenon.

The delay in the loss of stability has been analyzed through different bifurcations including the Hopf bifurcation \cite{N87,N88,BER89,ERHG91,HKSW16}, and as of particular interest for this paper, the transcritical bifurcation \cite{krupa2001extending,H00,EK19, K23,L75, L77, H79, D94}. Transcritical singularities, have been studied by means of different methods including the blow up technique \cite{krupa2001extending}, the discretization of the flow \cite{EK19} and more recently by local linearization along the critical manifold \cite{K23}.

In this paper, we approach the slow passage through a transcritical bifurcation by constructing a piecewise linear (PWL) minimal system. PWL systems have been widely used to reproduce non linear behaviour and to provide complementary understanding about some bifurcations \cite{llibre2007horseshoes, carmona2008existence, ponce2022bifurcations}. In addition, the PWL approach has recently been used to study slow-fast dynamics. Indeed, approaching slow-fast systems by PWL systems has a significant advantage: it straightforwardly provides canonical slow manifolds, which can be explicitly computed. Since the slow manifold is an essential part of the skeleton of slow-fast dynamics, the PWL approach allows for a treatable but still meaningful problem as it does not miss the
salient features of their smooth counterpart (see for example \cite{FDKT15,DGPPRT16,DFKPT18,CFT20,CFT23}). Indeed, the PWL approach has been used recently to study delayed bifurcations, as for example the Hopf and the Homoclinic bifurcations \cite{PDTV22,PDTV23}.

When studying the delay of the loss of stability, an important function quantifying such delay is the so-called way-in/way-out function. Consider a tubular $\delta$-neighbourhood $N_\delta$ around the slow manifold. Given an orbit, the way-in/way-out function relates the distance between the bifurcation point and the inner point of the orbit through $N_\delta$, with the distance between the bifurcation point and the outer point of the orbit through $N_\delta$. Roughly speaking, the way-in/way-out function relates the ratio of repulsion and attraction of the branches of the slow manifold. The asymptotic value of the way-in/way-out function is known as the maximal delay. 

By analysing the slow passage through a transcritical bifurcation through the PWL framework, we achieve a full control on the way-in/way-out function and on the maximal delay. By means of this understanding we describe well-known situations as, for example, how the maximal delay goes to zero as the singular parameter goes to zero. However, we also tackle non trivial behaviours as for instance the degenerate situation in which the maximal delay is unbounded. Moreover, our results also cover the transitional regime from the trivial to the degenerate case. In this intermediate  regime the maximal delay is finite when the singular parameter tends to zero. 

As we discussed before, this transitional regime is related with canard trajectories. In  \cite{FPV08}, smooth slow-fast systems with two (or more) transcritical bifurcations are constructed in such a way one can have canard cycles allowing to study canard regimes and enhanced delay. Following the ideas introduced in \cite{FPV08}, we combine two of our PWL units exhibiting slow passage through a transcritical bifurcation. As a result, the orbits leaving from the repelling slow manifold of one unit connect with the stable slow manifold of the other one. In this way, we generate oscillatory behaviour by alternating the passage through one PWL unit to the other. 
In this scenario we show the existence of a one parameter family of limit cycles undergoing a canard explosion up to a critical value in which the amplitude of the canard cycles tends to infinity. When the parameter takes such a critical value, we also show the occurrence of the enhanced delay phenomenon in which each passage through each unit increases the delay of the next passage.

Our paper is organised as follows: in Section~\ref{sec:section-2}, we present our PWL minimal model for the transcritical bifurcation and provide expressions for the canonical slow manifolds. In Section~\ref{sec:mainrsults}, we state the main results of the manuscript. In Section~\ref{sec:canard-explosion} we build a PWL system exhibiting two transcritical bifurcations and apply our previous results to analyse the canard explosions and the enhanced delay occurring at this system. We encapsulate our findings in Section~\ref{sec:conclusions}. In Section~\ref{sec:paper-proofs} we provide the proof of the main results. The paper ends with appendix \ref{sec:appendix} where we provide the local expressions for the flows of our PWL minimal model.





\section{A PWL minimal model for a transcritical bifurcation}\label{sec:section-2}

Let us consider the continuous piecewise linear differential  system
\begin{equation}\label{eq:transBif}
\begin{aligned}
x'&=|x|-|y|+\lambda\varepsilon,\\
y'&=\varepsilon,
\end{aligned}
\end{equation}
with parameters $0\leq \varepsilon \ll 1$ and $\lambda>0.$ The Lipschitz character of the vector field ensures the uniqueness of the solutions of the initial condition problem. Hence, let $\varphi(t;(x_0,y_0),\lambda,\varepsilon)$ denote the solution of  \eqref{eq:transBif} with initial condition $(x_0,y_0)$. Even if System \eqref{eq:transBif} is globally non-linear, when restricted to any of the four quadrants, $Q_k$ with $k\in\{1,2,3,4\}$, it becomes linear. Therefore, local expressions for the flow can be easily obtained, see Eq.~\eqref{eq:theFlows-x}-\eqref{eq:theFlows-y}.

For $\varepsilon = 0$ the equilibrium points of System \eqref{eq:transBif} define the critical manifold
\begin{equation}\label{eq:critical-mainfold}
\mathcal{S}_0=\{(x,y): ~|x|=|y|\}.
\end{equation}
All the equilibrium points in $\mathcal{S}_0$, except the origin, are normally hyperbolic. Indeed, its respective Jacobian matrix has an eigenvalue with real part equal to zero and another eigenvalue equal to $\pm 1$. Hence, the critical manifold $\mathcal{S}_0$, is the union of: (i) four normally hyperbolic branches, two attracting $\mathcal{S}^\pm_{a,0}$ and two repelling $\mathcal{S}^\pm_{r,0}$
\begin{equation}\label{eq:critical-mainfold-branches}
\mathcal{S}_{a,0}^\pm=\left\{(x,y):~y=\mp x,\enskip x<0\right\}, \qquad \mathcal{S}_{r,0}^\pm=\left\{(x,y):~y=\pm x,\enskip x<0\right\};
\end{equation}
and (ii) the point $(0,0)$ at which the normal hyperbolicity is lost since the vector field at the origin is not differentiable. We note that the subscript $a,r$ in the slow manifold indicates attracting or repelling and the superscript $\pm$ corresponds to the sign of the $y$ variable.

Let us now study the perturbed system  ($\varepsilon > 0$). In this case, Fenichel's theory \cite{F79,J95} implies that outside of a small neighbourhood of $(0,0)$, the different branches of the critical manifold persist as slow manifolds. By computing the eigenspace associated with the slow eigenvalue of System \eqref{eq:transBif}, we can obtain analytical expressions for  canonical slow manifolds \cite{PTV16}. In that way we obtain the following  expressions for the two canonical stable slow manifolds
\begin{equation}\label{def:atrslm}
\begin{array}{l}
\mathcal{S}_{a,\varepsilon}^-=\left\{(x,y):~y=x+\varepsilon(1-\lambda),\enskip~~x<\min\{0,\varepsilon(\lambda-1)\}\right\},\\
\mathcal{S}_{a,\varepsilon}^+=\{(x,y):~y=-x+\varepsilon(1+\lambda),~x<0 \}.
\end{array}
\end{equation}
and for the two canonical unstable slow manifolds
\begin{equation}\label{def:rpslm}
\begin{array}{l}
\mathcal{S}_{r,\varepsilon}^-=\{(x,y):~y=-x-\varepsilon(1+\lambda),~x<0\},\\
\mathcal{S}_{r,\varepsilon}^+=\left\{(x,y):~y=x-\varepsilon(1-\lambda),\enskip~~x>\max\{0,\varepsilon(1-\lambda)\} \right\},
\end{array}
\end{equation}
each of them perturbing from its respective branch of the critical manifold given in \eqref{eq:critical-mainfold-branches} (see Fig.~\ref{fig:my_label}).

\begin{figure}[hb]
    \centering
    \includegraphics[width=\textwidth]{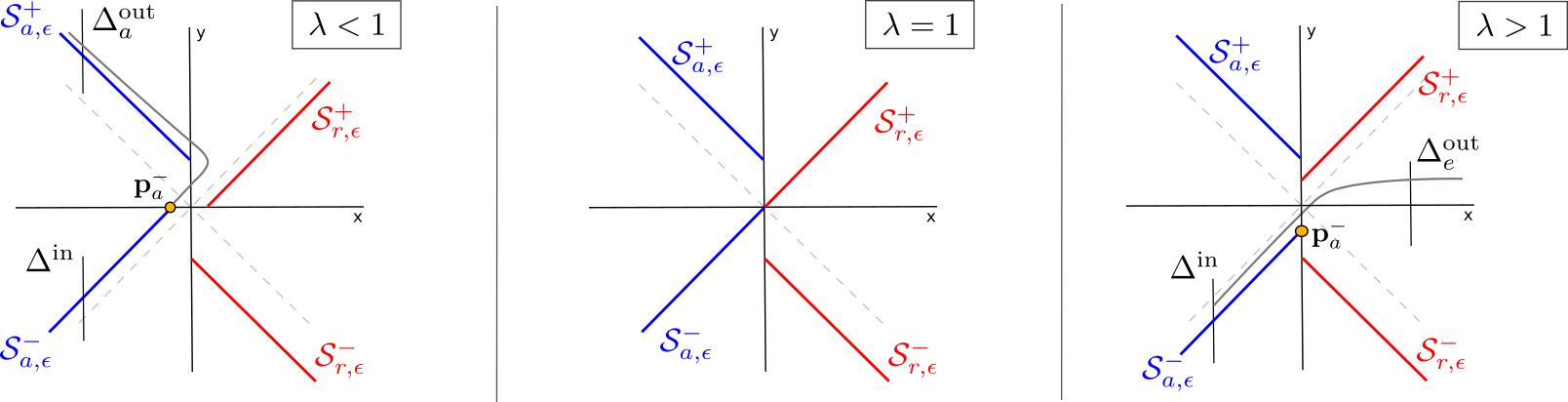}
    \caption{\textbf{Configuration of the slow manifolds}. Representation of the attracting slow manifold $\mathcal{S}_{a,\varepsilon}^{\pm}$ (blue lines) and the repelling slow manifold $\mathcal{S}_{r,\varepsilon}^{\pm}$ (red lines) of System \eqref{eq:transBif} depending on the values of $\lambda$, see expressions \eqref{def:atrslm}-\eqref{def:rpslm}. For $\lambda \neq 1$, we represent the crossing sections in \eqref{eq:crossing-sections} and sketch the evolution of the orbit in $\mathcal{S}_{a,\varepsilon}^{-}$ beyond the end point $\mathbf{p}_a^-$ given in \eqref{eq:pa-}.}
    \label{fig:my_label}
\end{figure}

\section{Statement of the main results}\label{sec:mainrsults}

In this section, for System \eqref{eq:transBif} and $\varepsilon >0$, we characterise how the trajectories in a small neighbourhood of the attracting slow manifold $\mathcal{S}_{a,\varepsilon}^-$ continue beyond the transcritical bifurcation at the origin. To this aim, we consider a tubular neighbourhood $N(\delta)$ of radius $\delta$, with $\delta>2\varepsilon$, around both the attracting $\mathcal{S}_{a,\varepsilon}^-$ and the repelling $\mathcal{S}_{r,\varepsilon}^+$ slow manifolds. By considering this tube, one can study the relationship between the point $\mathbf{p}_i=(x_i,y_i)^T$ in which an orbit $\gamma$ enters $N(\delta)$  and the point  $\mathbf{p}_o=(x_o,y_o)^T$ at which this orbit leaves $N(\delta)$. Indeed, the relationship $y_o$ versus $|y_i|$ defines a function $y_0 = \mathcal{Z}(|y_i|)$, usually denoted as the way-in/way-out function. Typically, the way-in/way-out function approaches an asymptotic value as $|y_i| \to \infty$. This value, which is a function of the parameters $\lambda, \varepsilon$, corresponds to the ordinate of the exit point $\mathbf{p}_a^-$ of the attracting slow manifold $\mathcal{S}_{a,\varepsilon}^-$ and is usually denoted as the \textit{maximal delay} $z_d(\lambda, \varepsilon)$. We quantify thorough the paper the delay in the loss of stability in terms of this maximal delay. 

Our results cover the cases $0<\lambda <1,~ \lambda = 1$ and $\lambda >1$. First, in \Cref{th:main1}, we offer a PWL perspective of the classical results by Krupa and Szmolyan in \cite{krupa2001extending} for fixed  $\lambda\neq1$. Additionally, in \Cref{th:main1delay} we interpret their results in terms of the maximal delay $z_d(\lambda, \varepsilon)$ and show that in these cases $\lim_{\varepsilon \searrow 0}z_d(\lambda, \varepsilon) = 0$. Moreover, in \Cref{th:main2}, we also study the degenerate case $\lambda = 1$ proving that for this value the maximal delay is unbounded. 

In order to bring together these two regimes we explore solutions where, instead of $\lambda$ being a fixed parameter, we consider $\lambda(\varepsilon) = 1 \pm e^{-c/\varepsilon}$ with $c \in \mathbb{R}^+$. For such a value of $\lambda$, in \Cref{th:main3}, we prove that  System~\eqref{eq:transBif}: (i) exhibits trajectories with canard segments; and (ii) its maximal delay satisfies $\lim_{\varepsilon \searrow 0}z_d(\lambda, \varepsilon) = c$ (see Fig.~\ref{fig:losDelays} and Remark~\ref{rm:exp-delays}).

\subsection{Slow passage for fixed $\lambda$}\label{sec:sec3-1}

Similarly as in \cite{krupa2001extending}, we define the crossing sections 
\begin{equation} \label{eq:crossing-sections}
\begin{array}{l}
\din=\{(-\rho,-\rho+\varepsilon(1-\lambda)+r): |r|<\delta\},\\
\dout_{a}=\{(-\rho,\rho+\varepsilon(1+\lambda)+r): |r|<\delta\},\\
\dout_{e}=\{(\rho,r): |r|<\delta\},
\end{array}
\end{equation}
where $\rho>0$ and $|\delta|\ll 1$ and the transition maps
\[
\Pi_e:\din \to \dout_{e},\quad \Pi_a:\din \to \dout_{a},
\]
from $\din$ to $\dout_{e}$ and $\dout_{a}$, respectively. Moreover, we define as $\mathbf{p}_a^-$ the point in the boundary of $\mathcal{S}_{a,\varepsilon}^-$ given by
\begin{equation}\label{eq:pa-}
    \mathbf{p}_a^-=(x,~x+\varepsilon(1-\lambda)), \quad \text{with~} x=\min\{0,~\varepsilon(\lambda-1)\},
\end{equation}
and $\gamma_{\mathbf{p}_a^-}$ as the orbit through $\mathbf{p}_a^-$. Therefore our PWL view of the results in \cite{krupa2001extending} are as follows:

\begin{theorem}\label{th:main1} Consider System \eqref{eq:transBif}. For fixed $\lambda >0$, there exists $\varepsilon_0>0$ such that the following statements hold for $\varepsilon\in (0,\varepsilon_0]$.
\begin{itemize}
\item [a)] If $\lambda>1$ then the orbit $\gamma_{\mathbf{p}_a^-}$ passes through $\dout_e$ at a point $(\rho, h(\varepsilon))$ where $h(\varepsilon)=O(\varepsilon\ln(\varepsilon))$. The section $\din$ is mapped by $\Pi_a$ to an interval containing $\gamma_{\mathbf{p}_a^-} \cap \dout_{e}$ of size $O(e^{-c/\varepsilon})$ where $c$ is a positive constant. 
\item [b)] If $\lambda <1$ then, $\din$ is mapped by $\Pi_e$ to an interval containing ${\mathcal{S}}_{a,\varepsilon}^{+}\cap \dout_{a}$ of size $O(e^{-c/\varepsilon})$ where $c$ is a positive constant.
\end{itemize}
\end{theorem}

\begin{remark}
The dynamical behaviour described in Theorem~\ref{th:main1} is completely comparable to that for the transcritical bifurcation in the smooth context  appearing in Theorem 2.1 in \cite{krupa2001extending}. Nevertheless, for $\lambda>1$, there is a quantitative difference in the size of the ordinate $h(\varepsilon)$, while in \cite{krupa2001extending} it is order $O(\sqrt{\varepsilon})$, here it is order $O(\varepsilon |\ln(\varepsilon)|)$. We note that in both cases $h(\varepsilon)$ tends to zero with $\varepsilon$. 
\end{remark}

Theorem~\ref{th:main1} describes the evolution of the slow manifold $\mathcal{S}_{a,\varepsilon}^-$, and of the orbits surrounding it,  beyond a neighbourhood of the origin, see Figure \ref{fig:my_label}. However, as we illustrate in Fig.~\ref{fig:figurica-xula}(a), it is also interesting to  describe this evolution in terms of the delay in the loss of stability, or more concretely in terms of the maximal delay $z_d(\lambda, \varepsilon)$. In the next result we board this question by revisiting Theorem~\ref{th:main1} in terms of the maximal delay for $\lambda\neq1$ but close enough to 1, see Remark~\ref{rem:los-logaritmos}.


\begin{figure}
    \centering
    \includegraphics[width=\textwidth]{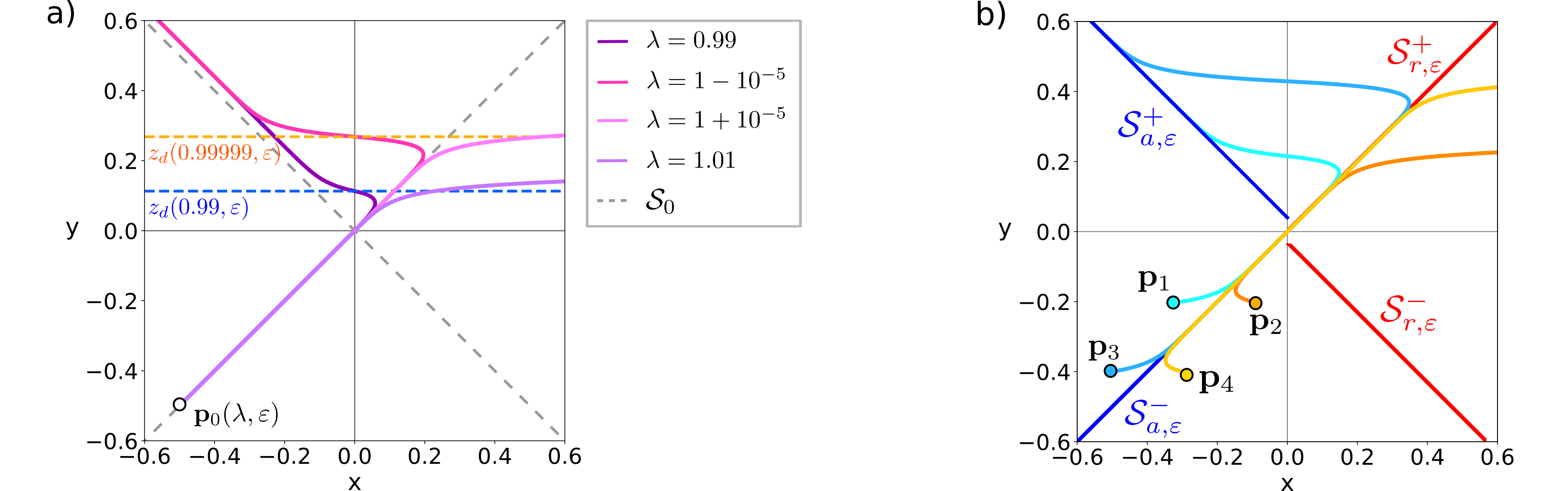}
    \caption{\textbf{Studying trajectories in terms of the maximal delay $z_d(\lambda,\varepsilon)$}. In panel (a) we fix $\varepsilon=0.02$ and consider a trajectory starting at a point $\mathbf{p}_0(\lambda, \varepsilon) = (x, x + \varepsilon(1-\lambda)) \in \mathcal{S}_{a,\varepsilon}^-$ for $\lambda \in \{0.99, 1-10^{-5},1+10^{-5},1.01\}$ and $x=-0.5$. As shown in the plot, and studied in Theorem~\ref{th:main1delay}, we observe how, for $\lambda = 1-\delta$, the crossing of the trajectory with the y-axis provides a boundary in the delay of the loss of stability which is valid for both for $\lambda = 1-\delta$ and $\lambda = 1+\delta$. In the panel (b) we consider $\varepsilon = 0.02$ and fix $\lambda=1$, so the manifolds $\mathcal{S}_{a,\varepsilon}^-$ and $\mathcal{S}_{r,\varepsilon}^+$ connect. We observe how the delay in the loss of stability increases with the distance of the initial condition $\mathbf{p}_k$ to the origin.}
    \label{fig:figurica-xula}
\end{figure}

\begin{theorem}\label{th:main1delay}
Consider System \eqref{eq:transBif} and the  differentiable function $C(\lambda)$ defined in \eqref{eq:f}. For fixed $1-\delta<\lambda <1+\delta $ with $\delta=1-\ln(2)$, there exists $\varepsilon_0>0$  such that the following statements hold for $\varepsilon\in (0,\varepsilon_0]$.
\begin{itemize}
    \item [a)] If $\lambda>1$, the maximal delay satisfies that $z_d(\lambda,\varepsilon)\leq \varepsilon(2-\lambda+C(2-\lambda))\leq O(\varepsilon\ln(|\varepsilon|))$.
    \item [b)] If $\lambda<1$, 
    the maximal delay satisfies that $z_d(\lambda, \varepsilon) \leq \varepsilon\,(\lambda+C(\lambda))$. 
\end{itemize}
\end{theorem}

\begin{remark}\label{rem:los-logaritmos} 
The value selected for parameter $\delta=1-\ln(2)$ in Theorem~\ref{th:main1delay}  is chosen for purely technical considerations and does not carry any relevant dynamic meaning within the context of the problem. This value is specifically determined to facilitate the proof of the assertion, without affecting the interpretation or validity of the results obtained.
\end{remark}

\begin{remark}\label{rem:zdcero}
 From Theorem~\ref{th:main1delay}, it follows that fixed $\lambda\neq1$,  $\lim_{\varepsilon \searrow 0}z_d(\lambda, \varepsilon) = 0$  (see Fig.~\ref{fig:losDelays}). 
\end{remark}


\begin{remark}\label{rem:casoc}
The value of $\varepsilon_0$ appearing in Theorem \ref{th:main1} and in Theorem \ref{th:main1delay} is a function of $\lambda$. Indeed, as we show in Section~\ref{sec:paper-proofs} (see Equations\eqref{def:eps0} and \eqref{def:ep1}), we obtain the following explicit expression for $\varepsilon_0(\lambda)$
\begin{equation}
    \varepsilon_0(\lambda)=\left\{
    \begin{array}{ll}
    \frac {e^{1-C(\lambda)}}{C(\lambda)-1} & \lambda<1,\\ 
    \min \left\{  \left(\frac {2(e^{\lambda-1}-1)}{\rho+\lambda-1} \right)^{\frac 1 2},
    \frac {\rho}{2e^{\lambda-1}-(\lambda+1)}\right\}  & \lambda>1,
    \end{array}
     \right.
\end{equation}
where $\rho \in \mathbb{R}$ and $C(\lambda)$ is a differentiable function. From this expression and since $\lim_{\lambda \searrow 1} C(\lambda)=+\infty$ (see Lemma~\ref{lem:exist_C}) it follows that $\varepsilon_0(\lambda)$ shrinks to $0$ as $\lambda$ tends to $1$. Therefore, $\lambda=1$ corresponds to a degenerate case. 
\end{remark}

In the next result, we study the behaviour of the flow surrounding $\mathcal{S}_{a,\varepsilon}^-$ and $\mathcal{S}_{r,\varepsilon}^+$ in the degenerate case $\lambda=1$ in which both manifolds connect (see also Fig.~\ref{fig:figurica-xula}(b)).

\begin{theorem}\label{th:main2} Consider System \eqref{eq:transBif}. For $\lambda=1$ and $\varepsilon>0$, the attracting branch $\mathcal{S}_{a,\varepsilon}^-$ and the repelling branch $\mathcal{S}_{r,\varepsilon}^+$ of the slow manifold connect at the origin. Therefore, the maximal delay is unbounded. Moreover, let $N(\delta)$ be a tubular neighbourhood of the slow manifold $\mathcal{S}_{a,\varepsilon}^-\cup \mathcal{S}_{r,\varepsilon}^+$ of radius $\delta$, then the graph of the way-in/way-out function is located between the lines $y=x+\delta$ and $y=x-\delta$. Consequently, the way-in/way-out function tends asymptotically to infinity and therefore, is not bounded.
\end{theorem}

\begin{remark}\label{rem:2Vidal}
In \cite{FV12}, the authors address the study of slow passage through the transcritical bifurcation using the differentiable normal form
\begin{equation}
\begin{aligned}
    \dot{x} &= -y x + x^2, \\
    \dot{y} &= \varepsilon.
\end{aligned}
\end{equation}
By choosing this system the authors in \cite{FV12} conclude that the attractive slow manifold connects with the repelling slow manifold. Thus, the situation presented in \cite{FV12} coincides with that stated in Theorem \ref{th:main2}. Although the delay analysis conducted in \cite{FV12} is presented in terms of the time variable, the result obtained is entirely equivalent to that obtained in Theorem \ref{th:main2}.
\end{remark}

\begin{remark}\label{rem:2}
    In  Remark 2.2 in  \cite{krupa2001extending}, the authors assure the existence of a function $\lambda_c(\sqrt{\varepsilon})$ with $\lambda_c(0)=1$ and such that, for $\lambda=\lambda_c(\sqrt{\varepsilon})$, the slow manifold $\mathcal{S}_{a,\varepsilon}^-$ extends to $\mathcal{S}_{r,\varepsilon}^+$ for $\varepsilon$ sufficiently small. From our Theorem \ref{th:main2}, it follows that in the PWL setup, the function $\lambda_c(\sqrt{\varepsilon})$ is identically $1$. Moreover, the authors in  \cite{krupa2001extending} claim that ``for values of $\lambda$ exponentially close to $\lambda_c(\sqrt{\varepsilon})$ the slow manifold $\mathcal{S}_{a,\varepsilon}^-$ can follow $\mathcal{S}_{r,\varepsilon}^+$ over an $O(1)$-distance before being repelled". This last statement motivates the next section.
\end{remark}

\subsection{Slow passage for $\lambda(\varepsilon)$}\label{sec:sec3-2}

In this section we study how the maximal delay $z_d(\lambda, \varepsilon)$ behaves for values of $\lambda$ exponentially close to the degenerate value $\lambda=1$. In particular we explore $\lambda(\varepsilon) = 1 \pm e^{-c/\varepsilon}$. 



\begin{theorem}\label{th:main3} 
Consider System \eqref{eq:transBif}. Let $c$ and $\varepsilon$ be positive constants and $\varepsilon$ small enough. 
\begin{itemize}
\item [a)] If $\lambda=1-e^{-\frac c{\varepsilon}}$, then for $\tau=\tau_1+\tau_2$ where $\tau_1=e^{-\frac c{\varepsilon}}+O\left(e^{-\frac {2c}{\varepsilon}}\right)$ and $\tau_2=\frac {c}{\varepsilon} + \ln(\frac {c}{2\varepsilon})$, it follows that $\varphi(\tau;\mathbf{p}_a^-)=(0+X(\varepsilon),c+Y(\varepsilon))^T$ where 
$X(\varepsilon)= \varepsilon \ln\left(\frac c {2\varepsilon} \right)+O(\varepsilon e^{-\frac c{\varepsilon}})$ and $Y(\varepsilon)=\varepsilon \ln\left(\frac c {2\varepsilon} \right) +O\left(e^{-\frac {c}{\varepsilon}}\right)$.
\item [b)] If $\lambda=1+e^{-\frac c{\varepsilon}}$, then for $\tau=\tau_1+\tau_2$ where $\tau_1=e^{-\frac c{\varepsilon}}$ and $\tau_2=\frac {c}{\varepsilon} + \ln(\frac {c}{2\varepsilon})$, it follows that $\varphi(\tau;\mathbf{p}_a^-)=(2c+X(\varepsilon),c+Y(\varepsilon))^T$ where 
$X(\varepsilon)= \varepsilon \ln\left(\frac c {2\varepsilon} \right)+O(\varepsilon e^{-\frac c{\varepsilon}})$ and $Y(\varepsilon)=\varepsilon \ln\left(\frac c {2\varepsilon} \right)$. 
\end{itemize}
\end{theorem}

\begin{remark}\label{rm:exp-delays}
As a consequence of Theorem~\ref{th:main3}, when $\lambda$ is $O(e^{-\frac{c}{\varepsilon}})$ close to $1$, the maximal delay satisfies $z_d(\lambda,\varepsilon)=c+h(\varepsilon)$ and $\lim_{\varepsilon \searrow 0}z_d(\lambda(\varepsilon),\varepsilon) = c$ (see Fig.~\ref{fig:losDelays}). 
\end{remark}

\begin{figure}[ht]
\begin{center}
\includegraphics[scale=0.5]{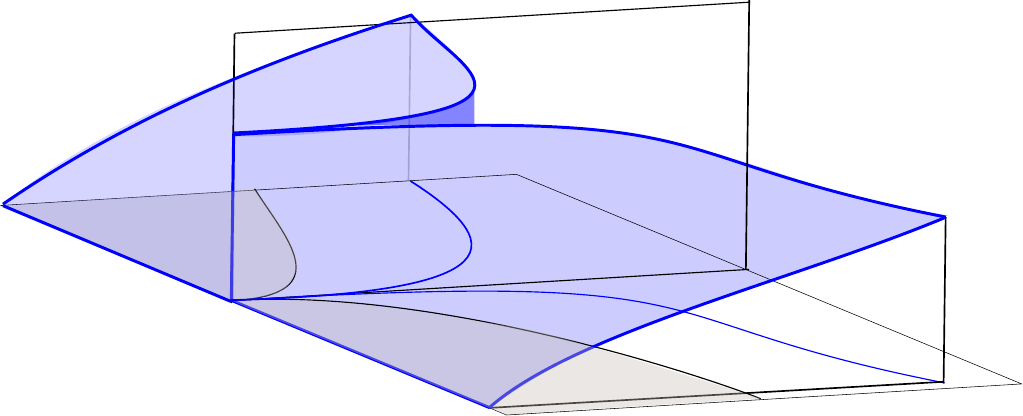}
\begin{picture}(0,0)
\put(-220,20){$\lambda=1$}
\put(-202,65){$c$}
\put(-60,60){$z_d(\lambda(\varepsilon),\varepsilon)\approx c$}
\put(-77,8){$\varepsilon_0(\lambda)$}
\put(-72,22){$\lambda(\varepsilon)$}
\end{picture}
\end{center}
\caption{\textbf{Sketch of the maximal delay function $z_d$ over the $(\lambda,\varepsilon)$-plane}.  The shadowed region corresponds to $\varepsilon<\varepsilon_0(\lambda)$ in Remark~\ref{rem:casoc} and shows that for fixed $\lambda \neq 1$, $\lim_{\varepsilon \searrow 0}z_d(\lambda,\varepsilon) = 0$. The blue curve on the domain corresponds to the parameter curve $\lambda(\varepsilon)=1\pm e^{-\frac{c}{\varepsilon}}$ showing that  $\lim_{\varepsilon \searrow 0}z_d(\lambda(\varepsilon),\varepsilon) = c$. }\label{fig:losDelays}
\end{figure}
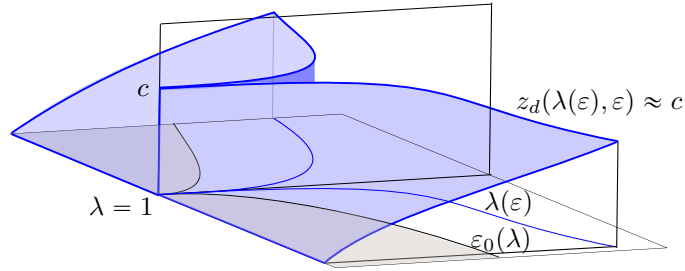



\section{Canard explosion and enhanced delay}\label{sec:canard-explosion}

In this section we study an application of the slow passage through a transcritical bifurcation inspired by the work by \cite{FPV08,FV12} in the smooth context. To that aim, we consider the two parameter $(\lambda,\varepsilon)$ PWL system given by
\begin{equation}\label{sys:2transc}
\left\{
\begin{array}{ll}
\left\{
 \begin{array}{l}
 \dot{x}=|x+1/2|-|y-1/2|+\lambda \varepsilon,\\
 \dot{y}=\varepsilon (y-x),
 \end{array}
\right. & y-x\geq 0,\\ \\ 
\left\{
 \begin{array}{l}
 \dot{x}=-|x-1/2|+|y+1/2|-\lambda \varepsilon,\\
 \dot{y}=\varepsilon (y-x),
 \end{array}
\right. & y-x\leq 0.
\end{array}
\right.
\end{equation}

As one can see, System \eqref{sys:2transc} consists in two different passages through transcritical bifurcations which are locally identical to the one we have previously analyzed in Section~\ref{sec:mainrsults}. Indeed, the two transcritical bifurcations are located at $\mathbf{p}_-=(-1/2,1/2)$ and $\mathbf{p}_+=(1/2,-1/2)$ which are on the straight lines $y-x = 1$ and $y-x = -1$, respectively. Therefore, as long as the trajectories remain close to these lines, System \eqref{sys:2transc} locally shows comparable dynamics as the one generated by the System \eqref{eq:transBif}. Consequently, we consider each passage separately and then we use the local dynamics to describe the global flow. As the relationship between the single units and the coupled system has not been formally established, from now on we depart from presenting the results in the form of theorems and instead compile them into remarks, which are validated by numerical simulations.

System  \eqref{sys:2transc}  is formed by two copies of the previously studied System \eqref{eq:transBif}, with a common boundary at the line $y-x=0$. Even if for $\varepsilon=0$ the vector field of System \eqref{sys:2transc} is continuous, that is nevermore the case when $\varepsilon>0$ because of the constant terms $\pm \lambda \varepsilon$ appearing in both subsystems. Because of this discontinuity, we follow Filippov's convention \cite{F13,GST11} to define the vector field over the line $y=x$ and hence, for $\varepsilon>0$ there exists a sliding segment $S_l$ along the switching line $y=x$. This sliding segment is limited by the end points  $\mathbf{e}_-= (-\frac {\lambda \varepsilon}{2} ,-\frac {\lambda \varepsilon}{2} )$ and $\mathbf{e}_+=(\frac {\lambda \varepsilon}{2} ,\frac {\lambda \varepsilon}{2} )$, corresponding to boundary equilibrium points of System \eqref{sys:2transc} in $y-x\geq 0$ and $y-x\leq 0$, respectively. Therefore, for every point of the phase plane but the sliding segment, there exists a unique orbit passing through it \cite{GST11}. We also note that as System \eqref{sys:2transc} is invariant under the change of variables $(x,y)\to (-x,-y)$, the orbits are symmetric with respect to the origin.

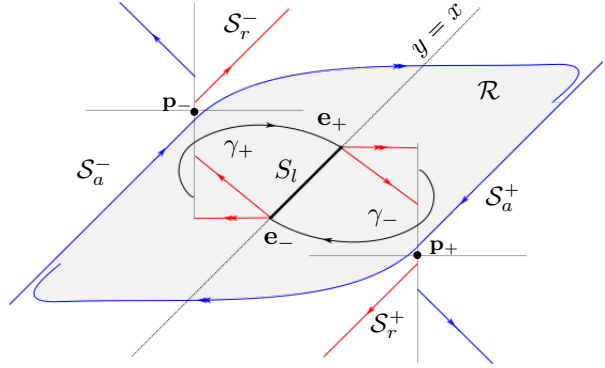
\begin{figure}[ht]
\begin{center}
\includegraphics[scale=0.5]{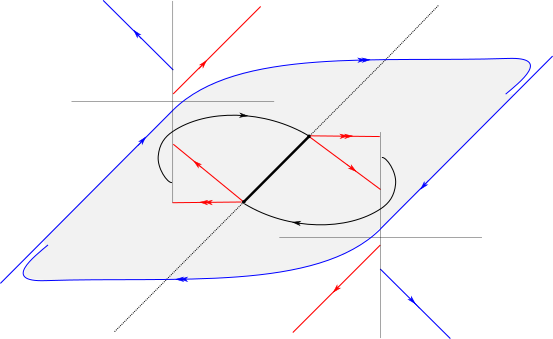}
\begin{picture}(0,0)
\put(-200,70){$\mathcal{S}_a^-$}
\put(-145,125){$\mathcal{S}_r^-$}
\put(-126,70){$S_l$}
\put(-130,45){$\mathbf{e}_-$}
\put(-110,90){$\mathbf{e}_+$}
\put(-91,54){$\gamma_-$}
\put(-145,80){$\gamma_+$}
\put(-168,97){\footnotesize$\mathbf{p}_-$}
\put(-156,95){\circle*{3}}
\put(-68,43){\footnotesize$\mathbf{p}_+$}
\put(-72.5,41){\circle*{3}}
\put(-47,60){$\mathcal{S}_a^+$}
\put(-90,12){$\mathcal{S}_r^+$}
\put(-50,100){$\mathcal{R}$}
\put(-75,116){\rotatebox{45}{$y=x$}}
\end{picture}
\end{center}
\caption{\textbf{Phase plane of System \eqref{sys:2transc}} containing the following elements: (i) the attracting slow manifolds $\mathcal{S}_a^{\pm}$ and the repelling slow manifolds $\mathcal{S}_r^{\pm}$ associated with the transcritical bifurcation at points $\mathbf{p}_-=(-\frac 1 2,  \frac 1 2)$ and $\mathbf{p}_+=(\frac 1 2, -\frac 1 2)$; (ii) the positive invariant region $\mathcal{R}$ (dashed area) which is  limited by the attracting slow manifolds and contains the sliding segment $S_l$; (iii) The boundary equilibrium points $\mathbf{e}_+$ and $\mathbf{e}_-$ which are of node type, together with the fast eigenspace and the slow eigenspace (indicated with two and one arrows, respectively).}\label{fig:applicat1}
\end{figure}  

The local analysis around the transcritical bifurcation (see Theorem~\ref{th:main1delay}(a)), ensures that when $\lambda>1$ the attracting slow manifold $\mathcal{S}_a^-$ of the subsystem located in $y-x>0$, leaves the neighbourhood of the repelling branch $\mathcal{S}_r^-$ at a distance $O(\varepsilon |\ln(\varepsilon)|)$ above of the bifurcation point $\mathbf{p}_- $. Then, because of the symmetry, the orbit gets a neighbourhood of the attracting branch $\mathcal{S}_a^+$ of the subsystem in $y-x<0$, passes through the bifurcation point $\mathbf{p}_+$, follows the repelling branch $\mathcal{S}_r^+$, and leaves it at a distance $O(\varepsilon |\ln(\varepsilon)|)$ below the point $\mathbf{p}_+$. This behaviour gives rise to a positive invariant region $\mathcal{R}$ containing the sliding segment $S_l$, see Figure~\ref{fig:applicat1}.


Let  $\gamma_-$ be the orbit locally contained in $y-x<0$ with $\mathbf{e}_-$ as $\omega-$limit set (resp. $\gamma_+$ the orbit locally contained in $y-x>0$ with $\mathbf{e}_+$ as $\omega-$limit set). We now discuss the behaviour of the orbits inside $\mathcal{R}$ depending on the parameters $\lambda, \varepsilon$. For $\lambda, \varepsilon$ values such that $\mathbf{e}_+$ is contained in the $\alpha-$limit set of $\gamma_-$, then by the symmetry of the flow it follows that $\mathbf{e}_-\in \alpha(\gamma_+)$. In consequence, there is a heteroclinic connection $\Gamma_h=\gamma_- \cup \gamma_+$ inside $\mathcal{R}$ and containing the sliding segment $S_l$. Otherwise, the sliding segment $S_l$ is a global attractor for the flow in $\mathcal{R}$, see Figure~\ref{fig:applicat1}. 

Let us now discuss about the stability of the  heteroclinic connection $\Gamma_h$. First, we note that $\mathbf{e}_-$ and $\mathbf{e}_+$ are locally repellor equilibrium points of node type, both having one fast eigenspace and one slow eigenspace, see Figure \ref{fig:applicat1}. For the critical values of the parameters $\lambda, \varepsilon$ for which $\gamma_+$ contains the fast eigenspace of $\mathbf{e}_-$ the heteroclinic cycle $\Gamma_h$ appears and it is outside stable and inside unstable. Conversely, when $\gamma_-$ does not contain the fast eigenspace of $\mathbf{e}_+$ then $\Gamma_h$ is unstable both inside and outside. Specifically, the $\alpha-$limit set of every orbit in a neighborhood of $\Gamma_h$ is $\{\mathbf{e}_+,\mathbf{e}_-\}$. Therefore, since $\mathcal{R}$ is a positive invariant region and $\{\mathbf{e}_+,\mathbf{e}_-\}$ are the $\alpha$-limit set of any orbit in a neighborhood containing $\Gamma_h$, from the Poincaré-Bendixon Theorem we conclude that a stable limit cycle $\Gamma$  surrounding $\Gamma_h$ appears in region $\mathcal{R}$. 

The limit cycle $\Gamma$ borns with an amplitude $A$ equal to the amplitude of $\Gamma_h$, that is, $A=O(\|\mathbf{e}_+-\mathbf{e}_-\|)=O(\varepsilon\lambda)$. Furthermore, as $\varepsilon$ decreases to zero, the limit cycle approaches the attracting slow manifold $S_{a}^-$ and hence $\Gamma$ tends to the boundary of region $\mathcal{R}$. Therefore, when $\lambda$ is fixed and greater that 1, the attracting branch of the slow manifold $\mathcal{S}_{a}^-$ in $x-y<0$, leaves the neighbourhood of the repelling branch $\mathcal{S}_r^-$ at a distance $z_d(\lambda,\varepsilon)=\varepsilon|\ln(\varepsilon)|$ from the bifurcation point $\mathbf{p}_-$. Therefore, the amplitude of $\Gamma$ is $A=1+O(\varepsilon |\ln(\varepsilon)|)$. Furthermore, if we take $\lambda =1+e^{-\frac c{\varepsilon}}$ the attracting branch $\mathcal{S}_{a}^-$ follows for a time the repelling branch $\mathcal{S}_r^-$, giving rise to a canard segment. Then $\mathcal{S}_{a}^-$ leaves the neighbourhood of $\mathcal{S}_r^-$ at a distance $z_d(\lambda, \varepsilon)=c+\varepsilon h(\varepsilon)$ from the bifurcation point, where $\varepsilon h(\varepsilon)$ tends to zero as $\varepsilon \to 0$ (see Theorem~\ref{th:main3}(b)). Consequently, $\Gamma$ is a (headless) canard cycle with amplitude 
\begin{equation}\label{exp:Avsc}
A=1+2c+2\varepsilon h(\varepsilon) =1-2\varepsilon\ln(\lambda-1)+2\varepsilon h(\varepsilon) ,
\end{equation} 
where we used $c=-\varepsilon\ln(\lambda-1)$. In Figure~\ref{fig:elCanard} we draw the amplitude in \eqref{exp:Avsc} for the canard explosion taking place at $\lambda=1$ for $\varepsilon=0.02$ and $\varepsilon=0.05$. \\

Expression \eqref{exp:Avsc} provides an explicit relation between the amplitude $A$ of the canard cycle and the value of the parameter $\lambda$ at which that canard cycle exists. Taking the derivative of \eqref{exp:Avsc} with respect to $\lambda$ and isolating $c$ in \eqref{exp:Avsc} in terms of $A$, we can calculate the range of variation of the amplitude with respect to the parameter 
\[
\frac {dA}{d\lambda} =-2\varepsilon e^{\frac {A-1}{2\varepsilon}-h(\varepsilon)},
\]
which is exponentially big.

\begin{figure}[b]
    \centering\includegraphics[scale=0.6]{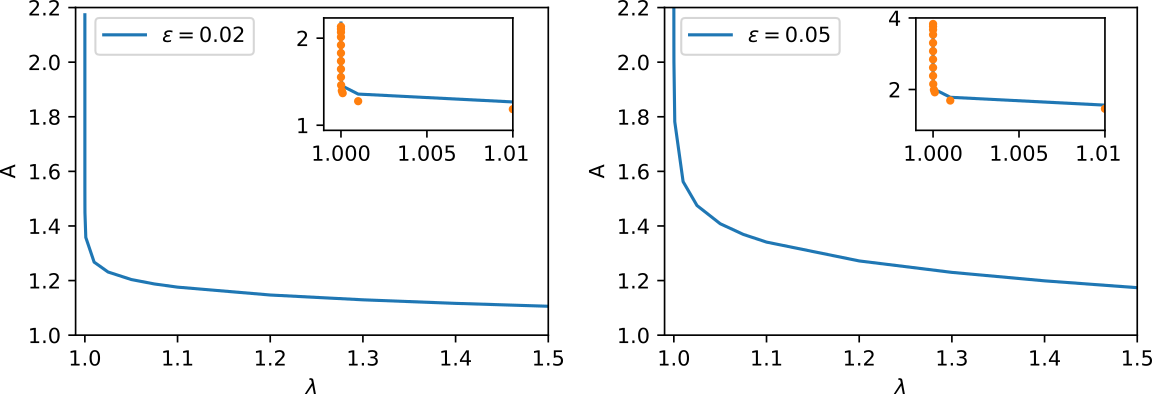}
    \caption{\textbf{Canard explosion occurring in System \eqref{sys:2transc}}. In this figure we represent the amplitude of the canard cycle versus the parameter $\lambda$ for $\varepsilon=0.02$ (left) and $\varepsilon=0.05$ (right). In both panels we draw an inset figure showing the computed amplitude $A$ of the canard explosion  (blue line) and the theoretical prediction, see Equation \eqref{exp:Avsc} (orange dots).}
    \label{fig:elCanard}
\end{figure}

\begin{remark}\label{rem:canardexpl}
We conclude that System \eqref{sys:2transc} presents a two parameter family of stable limit cycles $\Gamma_{\varepsilon,\lambda}$ for $\varepsilon>0$ and $\lambda>1$. Let $A(\varepsilon,\lambda)$ denote the amplitude of $\Gamma_{\varepsilon,\lambda}$.

Fixed $\varepsilon_0$ positive and small enough, the family of limit cycles appears for $\lambda_b=O({\varepsilon}_0^{-1})$ at a bifurcation at which the sliding segment $S_l$ changes its stability. The amplitude at the bifurcation is $A(\varepsilon_0,\lambda_b)=O(1)$. As $\lambda$ decreases to $1$ the amplitude of the limit cycles increases and the family overcomes a canard explosion, see Figure~\ref{fig:elCanard}. At the canard explosion, from \eqref{exp:Avsc}, the amplitude of the canard cycles can be accurately approximated by $A(\varepsilon_0,\lambda)=1-2\varepsilon_0 \ln (\lambda-1)$ (see the orange points in the inset panels of Figure~\ref{fig:elCanard}). From the previous expression of the amplitude we obtain the partial derivative $\frac {\partial A}{\partial \lambda}(\varepsilon_0,\lambda)=-\frac {2\varepsilon_0}{\lambda-1}$. Hence, for $\lambda$ exponentially close to $1$, that is $\lambda=1+e^{-\frac c {\varepsilon_0}}$, it follows that the amplitude of the canard cycle remains constant $A(\varepsilon_0,\lambda)=1+2c$. Finally, its derivative $\frac {\partial A}{\partial \lambda}(\varepsilon_0,\lambda)=-2\varepsilon e^{\frac c {\varepsilon}}$ is exponentially big. 

On the other side, fixed $\lambda_0>1$, as $\varepsilon$ decreases to zero, the amplitude of the limit cycles satisfies $A(\varepsilon,\lambda_0)=1+O(\varepsilon \ln(\varepsilon))$ where no dependence on $\lambda_0$ appears. The accuracy of this approximation is represented in Figure~\ref{fig:implosion}. 
\end{remark}

\begin{figure}[ht]
\centering
\includegraphics[scale=0.3]{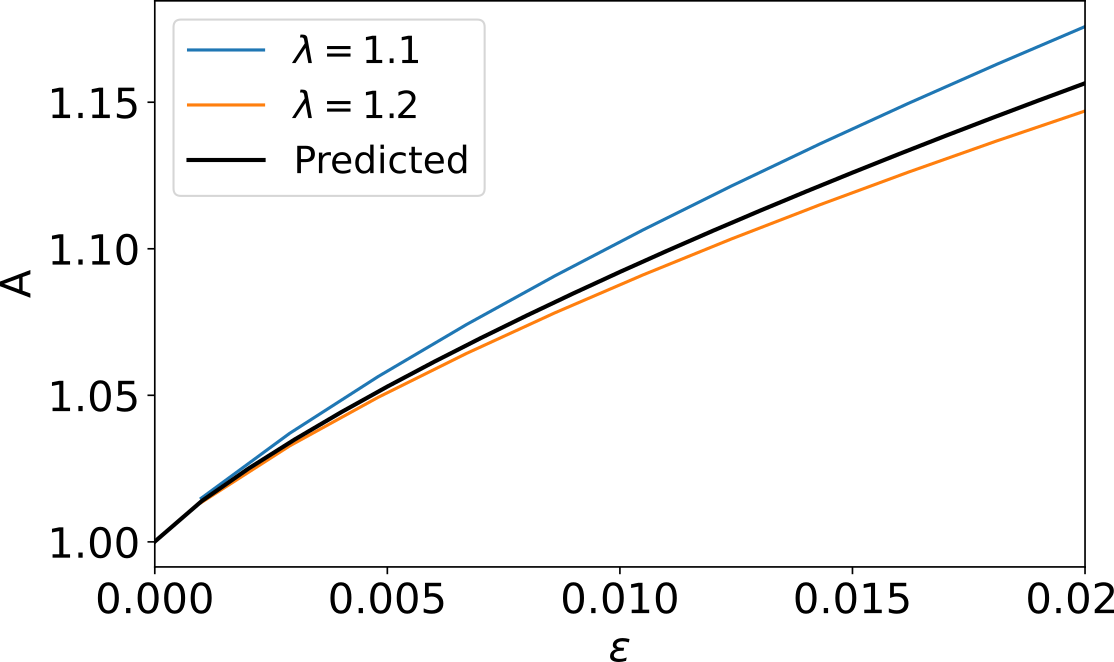}
\caption{\textbf{Amplitude of the cycles for fixed $\lambda$.} In this figure we show the validity of the theoretical approximation for the amplitude given by $A=1+2\varepsilon|\ln(\varepsilon)|$ for two fixed values of $\lambda$.} \label{fig:implosion}
\end{figure}

When $\lambda=1$, as a consequence of Theorem~\ref{th:main2}, it follows that the attracting slow manifold $\mathcal{S}_{a}^-$ (resp. $\mathcal{S}_{a}^+$) connects with the repelling slow manifold $\mathcal{S}_{r}^-$ (resp. $\mathcal{S}_{r}^+$) at the transcritical bifurcation point located at $\mathbf{p}_-=(-\frac 1 2, \frac 1 2)$ (resp. $\mathbf{p}_+=(\frac 1 2, -\frac 1 2)$ ).
Therefore, the previous mentioned invariant region $\mathcal{R}$ becomes an unbounded invariant strip in the phase plane given by $ -1 \leq y-x \leq  1$, that is, the region laterally limited by the union of the slow manifolds $\mathcal{S}_{a}^+\cup\mathcal{S}_{r}^+$ (and also $\mathcal{S}_{a}^-\cup\mathcal{S}_{r}^-$). Moreover, since the amplitude of canard cycles tends to infinity as $\lambda$ tends to $1^+$, we conclude that the canard explosion ends at a heteroclinic loop with the equilibrium points at infinity and with heteroclinic orbits given by  $\mathcal{S}_{a}^+\cup\mathbf{p}_+\cup\mathcal{S}_{r}^+$ and $\mathcal{S}_{a}^-\cup\mathbf{p}_-\cup\mathcal{S}_{r}^-$.

As we now discuss, our understanding of System \eqref{eq:transBif} allows us to provide theoretical predictions for the increase in each passage of the amplitude of the oscillatory behaviour existing in the invariant strip $\mathcal{R}$. Consider small amplitude tubular neighbourhoods along these four invariant manifolds and let us describe the dynamical behaviour of the orbits with respect to these tubular neighbourhoods. Consider an orbit which enters the tubular neighbourhood around $\mathcal{S}_{a}^-$ at a point with ordinate $y_0<\frac 1 2$. Then, Theorem~\ref{th:main2} ensures that the orbit suffers a delay of magnitude $\frac 1 2 - y_0$ after passing the transcritical bifurcation point, and leaves the neighbourhood of the repelling slow manifold $\mathcal{S}_{r}^-$ at a point with ordinate $1-y_0$. In a similar way, when an orbit gets inside the tubular neighbourhood around $\mathcal{S}_{a}^+$ with ordinate $y_0$, it suffers a delay $\frac 1 2+y_0$ and leaves the neighbourhood of $\mathcal{S}_{r}^-$ with ordinate $-1-y_0$. In particular, for $\varepsilon$ small enough, if an orbit approaches $\mathcal{S}_{a}^-$ with ordinate $y_0$, then it suffers  a first delay and leaves $\mathcal{S}_{r}^-$ at $y_1=1-y_0$, then approaching $\mathcal{S}_{a}^+$ with the same ordinate and, after a second delay, leaving $\mathcal{S}_{r}^+$ with ordinate $y_2=-1-y_1$, and then approaching again $\mathcal{S}_{a}^-$ but now with ordinate $y_2 = -2 + y_0$. In this case, the orbit suffers a delay which is $2$ units greater than the first delay. This phenomenon is called enhanced delay \cite{FPV08}. Finally, the orbit leaves the unstable slow manifold $\mathcal{S}_{r}^-$ with ordinate $y_3=1-y_2=3-y_0$, and so on. 

\begin{figure}[h]
    \centering\includegraphics[width=1\textwidth]{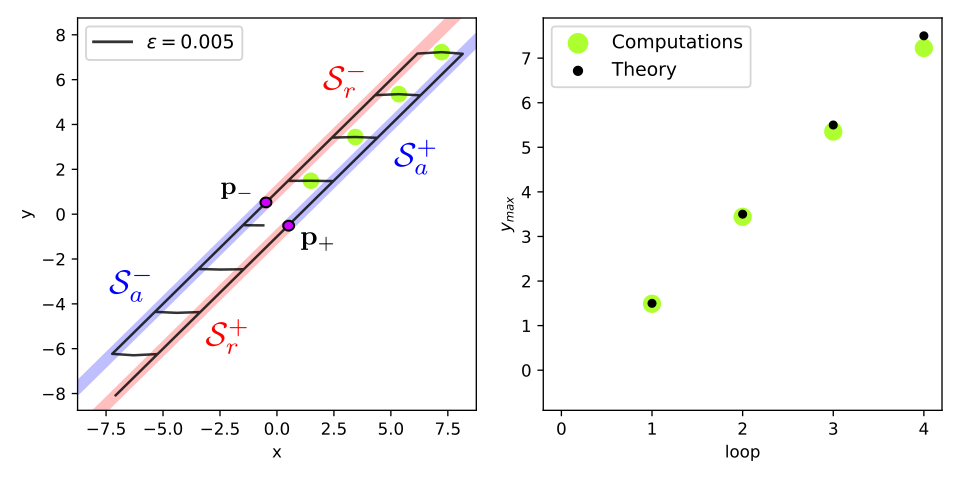}
    \caption{\textbf{Enhanced delay}. In this figure we illustrate the phenomenon of enhanced delay for System \eqref{sys:2transc}. In the left panel, together with both the bifurcation points and the singular manifolds sketched in Fig.~\ref{fig:applicat1}, we start a trajectory at $(x_0, y_0) = (-0.6, -0.5)$ and integrate numerically System \eqref{sys:2transc} for a time $t\approx14001.12$ (in which it completes 4 complete loops). We observe how the maximum value for the $y$ component (green dots) increases every loop thus showing the validity of the theoretical predictions in Remark~\ref{rm:enhDelay}. Indeed we computed $y_{\max} = \{1.49, 3.44, 5.35, 7.23 \}$ (to be compared with $\{1.5, 3.5, 5.5, 7.5 \}$). The accuracy of the theoretical prediction will increase as $\varepsilon \to 0$, however the numerical accuracy required checking this scenario is beyond our software.} 
    \label{fig:enhDelay}
\end{figure}

\begin{remark}\label{rm:enhDelay}
We conclude that, for sufficiently small $\varepsilon$, when an orbit approaches the attracting manifold with ordinate $y_0$, at each passage through one transcritical bifurcation it undergoes an enhanced delay of magnitude equal 2 units, giving raise to an oscillatory dynamical behaviour of increasing amplitude (see Fig.~\ref{fig:enhDelay} left panel), and whose peaks and valleys follow the sequences $\{2k-1-y_0\}_{k=1}^{\infty}$ , $\{-2k+y_0\}_{k=1}^{\infty}$, respectively (Fig.~\ref{fig:enhDelay} right panel).  
\end{remark}

\section{Conclusions}\label{sec:conclusions}

In this paper, we explore the slow passage phenomenon through a transcritical bifurcation in a minimal piecewise-linear (PWL) differential system.
Our study follows from previous studies in the smooth context. While Krupa and Szmolyan studied in \cite{krupa2001extending} the case in which the distance between the attracting and repelling slow manifolds is $O(\varepsilon)$, Françoise, Piquet and Vidal in \cite{FPV08}, and Vidal and Françoise in \cite{FV12} addressed the scenario where both manifolds connect. Besides considering these situations in the PWL context, we also tackled the case in which the manifolds are exponentially close, which can be regarded as the regime connecting both scenarios.
 
Our study benefits from the PWL framework as it allows to obtain explicit expressions of the slow manifolds and of the flow in their vicinity. Having such control enables explicit calculations of relevant system features, in contrast to the challenges posed by smooth systems. Specifically,  we ascertain the maximum delay $z_{d}(\lambda,\varepsilon)$ as a function of the system parameters  hence identifying three different behaviours which we analyse in Section \ref{sec:mainrsults}.  First, we delineate the parameter space region where the delay tends trivially to zero with $\varepsilon$, as indicated by Theorem \ref{th:main1delay}, Remark \ref{rem:zdcero}, and Remark \ref{rem:casoc}. These results are qualitatively (and almost quantitatively) equivalent to the ones described in \cite{krupa2001extending}  who considered a fixed value of $\lambda \neq 1$. Furthermore, we pinpoint the curve where the delay exhibits singular (unbounded) behavior, as outlined in Theorem \ref{th:main2}. Notably, in this configuration, the dynamics of System \eqref{eq:transBif} around the slow manifold resemble that observed in the smooth context in \cite{FV12}, as discussed in Remark \ref{rem:2Vidal}. Lastly, we identify the parameter region where the maximum delay exhibits transitional behavior, converging to a finite positive value with $\varepsilon$, as demonstrated in Theorem \ref{th:main3} and Remark \ref{rm:exp-delays}.

Inspired by the study on enhanced delay conducted in \cite{FV12}, in Section 4 we introduce a differentiable system with central symmetry and composed by two PWL systems exhibiting a transcritical bifurcation. Locally, in the vicinity of the critical manifolds, the flow of each of these systems  can be approximated by the behavior of the flow analyzed in Section 3. This approximation enables us to describe the global flow of the system approximately. Consequently, we find that System \eqref{sys:2transc} displays an unbounded canard explosion, that is, the amplitude of the canard cycles tends to infinity at a critical value. In fact, at this critical value the flow exhibits the phenomenon of enhanced delay described in \cite{FV12}. Again, the PWL framework allows us to explicitly quantify some of the critical quantities associated with the enhanced delay phenomenon. Specifically, the bifurcation point where limit cycles originate, the relationship between their amplitudes and the parameters (both in canard and non-canard regimes, see Remark \ref{rem:canardexpl}), and the increasing in the amplitude of each oscillation occurring in the enhanced delay phenomenon, see Remark \ref{rm:enhDelay}. We emphasize that all the obtained results on enhanced delay are consistent with those presented in \cite{FPV08, FV12} in the smooth context.

The results in this manuscript can be extended in different ways. In \cite{FV12}, the authors address the question of the structural stability of the enhanced delay. Indeed, introducing a perturbative term $\alpha n(t)$ to the first equation of System \eqref{eq:transBif} can be interpreted as perturbing the parameter $\lambda$ in the form $\left(\lambda+\frac{\alpha}{\varepsilon} n(t)\right)$. Thus, in the singular case $\lambda=1$, System \eqref{eq:transBif} transitions from the situation described in Theorem \ref{th:main1}(a) to that described in Theorem \ref{th:main1}(b) and vice versa. This suggests that even in the PWL case, the enhanced delay is a structurally unstable phenomenon. However, addressing properly this question is beyond the scope of this manuscript and appears as an interesting topic for future work. Additionally, and following  further results by Krupa and Szmolyan in \cite{krupa2001extending}, studying a minimal PWL system exhibiting the slow passage phenomenon through a pitchfork bifurcation seems also a natural way of extending the ideas and techniques outlined in this manuscript.



\section{Proof of the main results}\label{sec:paper-proofs}
In this section, we provide the proof of the main results of our study. In these proofs, the local expression of the solution $\varphi(t;(x_0,y_0),\lambda,\varepsilon)$ of System \eqref{eq:transBif} given in \eqref{eq:theFlows-x}-\eqref{eq:theFlows-y} plays a central role.
However, throughout this section, we omit the dependence of $\varphi$ in $\lambda, \varepsilon$ to maintain notation as simpler as possible.

\subsection{Proof of \Cref{th:main1}}

In the next lemma we describe the evolution of the orbit passing through the point 
\begin{equation}\label{eq:pt}
\mathbf{p}_t=(0,\varepsilon \lambda),
\end{equation}
where the flow of System \eqref{eq:transBif} is tangent to the line $\{x=0\}$. Hence, as we geometrically show in \Cref{fig:lm1.smconnection}(a) and next prove, this orbit delimits the set of orbits that intersect the quadrant $Q_1$ from that which do not.

\begin{lemma}\label{lem:pt_behav}
Let $\varphi(t;\mathbf{p}_t)$ be the solution of System~\eqref{eq:transBif} with initial condition $\mathbf{p}_t$.
\begin{itemize}
\item [a)] For $t\geq -\lambda$, the solution $\varphi(t;\mathbf{p}_t)$ is contained in $Q_4$, and $\varphi(-\lambda;\mathbf{p}_t)=(\varepsilon(\lambda-1)-\varepsilon (e^{\lambda}-2),0)^T$.
\item [b)] For $\tau_t = \frac {\rho}{\varepsilon}+1$ then, $d(\varphi(\tau_t;\mathbf{p}_t), {\mathcal{S}}_{a,\varepsilon}^{+} \cap \dout_a) = O(e^{-\frac {\rho}{\varepsilon}})$.
\end{itemize}
\end{lemma}
\begin{proof}
The solution $\varphi(t;\mathbf{p}_t)$ is tangent to $\{x=0\}$ at $\mathbf{p}_t$. In fact,
\[
\varphi(t;\mathbf{p}_t)=\mathbf{p}_t+\begin{pmatrix} 0\\\varepsilon \end{pmatrix} t + \begin{pmatrix} -\varepsilon \\0 \end{pmatrix} \frac {t^2}2 + O(t^3).
\]
Therefore, $\varphi(t;\mathbf{p}_t)$ is locally contained in $Q_4$. The time derivative of the first coordinate is negative for $t>0$, since the second coordinate is always increasing, we conclude that the solution is contained in $Q_4$ for $t\geq 0$. Similarly, for $t<0$ the first and the second coordinates of the solution decreases. As long as $t> -\lambda$, the second coordinate remains positive, and hence the corresponding part
of the orbit will be contained in $Q_4$. To finish the proof of the statement (a) we compute $\varphi(-\lambda; \mathbf{p}_t)$ straightforwardly from the expression of the flow in $Q_4$ given in  \eqref{eq:theFlows-x}-\eqref{eq:theFlows-y}. The proof of the statement (b) follows directly from evaluating the flow in $Q_4$ at the corresponding times. 
\end{proof}

\subsubsection{Case $\lambda < 1$}

\begin{figure}[h]
    \centering
    \includegraphics[width=\textwidth]{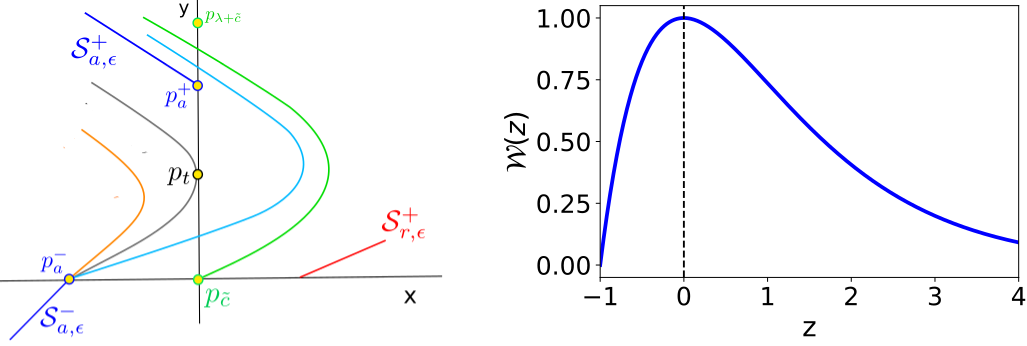}
    \caption{\textbf{Sketch of the possible behaviours of the flow through $\mathbf{p}_a^-$ for $0 < \lambda < 1$}. (Left) The orbit through $\mathbf{p}_t$ (grey curve for $\lambda = \ln(2)$) separates the orbits which are fully contained in $Q_4$ (orange curve for $\lambda \leq \ln(2)$) and the ones that also visit $Q_1$ (cyan curve for $ \ln(2) < \lambda < 1$). The trajectories through the origin (green curve), also determine the maximal height at which it is possible to cross the $y$-axis. (Right) Graph of the map $\mathcal{W}(z)$ on the domain $[-1, \infty$).}\label{fig:lm1.smconnection}

    \label{fig:lemmas}
\end{figure}

First, we enunciate the following lemmas that will help us to determine, for a given $\lambda<1$ value, the possible paths of the orbit starting at the point $\mathbf{p}_a^- = (\varepsilon(\lambda-1), 0)^\top $ given in \eqref{eq:pa-}. This orbit corresponds to the evolution of the attracting slow manifold $\mathcal{S}_{a,\varepsilon}^-$.

\begin{lemma}\label{lm:asm_behav_a}  Let $\varphi(t;\mathbf{p}_a^-)$ be the solution of System \eqref{eq:transBif} with initial condition at $\mathbf{p}_a^-$ and assume $\lambda\in (0,\ln(2)]$. 
\begin{itemize}
\item [a)] For $t\geq 0$, the solution $\varphi(t;\mathbf{p}_a^-)$ is contained in the quadrant $Q_4$.
\item [b)] For $\tau_a=\frac {\rho}{\varepsilon}+1+\lambda$, then $d(\varphi(\tau_a,\mathbf{p}_a^-),\mathcal{S}_{a,\varepsilon}^+ \cap \dout_a)=O(e^{-\frac {\rho}{\varepsilon}})$.
\end{itemize}
\end{lemma}
\begin{proof}
When $\lambda\in (0,\ln(2)]$ the first coordinate of $\mathbf{p}_a^-$ is less than the first coordinate of $\varphi(-\lambda; \mathbf{p}_t)$ see Lemma~\ref{lem:pt_behav}(a). Hence, the statement (a) follows as $\varphi(t; \mathbf{p}_t)$ is contained in $Q_4$. The statement (b) follows straightforwardly from evaluating the flow in $Q_4$ at $t=\tau_a$.
\end{proof}

In Lemma~\ref{lm:asm_behav_a} we have shown that for $0<\lambda \leq \ln(2)$, the solutions through $\mathbf{p}_a^-$ is contained in $Q_4$. 
Otherwise, for $\ln(2) < \lambda < 1$ the solution leaves $Q_4$, and evolves through $Q_1$ before entering again $Q_4$ by crossing the $y$-axis, see Lemma~\ref{lm:asm_behav_b}. In the next lemma, we provide an upper bound for the ordinate of such crossing. To that aim, we first define the functions $\mathcal{W}(z)$ and $C(\lambda)$ given by 
\begin{equation}\label{eq:f}
\mathcal{W}(z):=(1+z)e^{-z} \quad\text{and}\quad \mathcal{W}(C(\lambda))=\mathcal{W}(-\lambda),
\end{equation}
where the graph of $\mathcal{W}$ on the domain $[-1,+\infty)$ is depicted in~\Cref{fig:lm1.smconnection}(b), and $C(\lambda)$ is implicitly well-defined and analytical in the domain $(0,1)$. We straightforwardly conclude that $C'(\lambda)>0$,  $C(\lambda)>1$ if $\lambda \geq \ln(2)$, $\lim_{\lambda \to 1} C(\lambda)=+\infty$, and $\lim_{\lambda \to 1} C'(\lambda) =+\infty$.

\begin{lemma}\label{lem:exist_C}
Consider $\lambda\in (0,1)$ and let $C(\lambda)$ be the function implicitly defined in \eqref{eq:f}. The solution of System \eqref{eq:transBif} with initial condition at the origin, $\varphi(t;\mathbf{0})$, is locally contained in $Q_1$ for $t\in (0,\lambda+C(\lambda))$ and 
$\varphi(\lambda+C(\lambda);\mathbf{0})=(0,\varepsilon(\lambda+C(\lambda))^\top$.
\end{lemma}
\begin{proof}
The vector field at the origin is given by $\varepsilon (\lambda,1)$, so that the solution with initial condition at the origin is contained in $Q_1$ for a sufficiently small time interval. From the local expression of the solution $\varphi(t;\mathbf{0})$ while in $Q_1$, see \eqref{eq:theFlows-x}, we obtain that $\varphi(t;\mathbf{0})=((\lambda-1)\varepsilon e^t+\varepsilon t-(\lambda-1)\varepsilon, \varepsilon t)$. Therefore, $\varphi(t,\mathbf{0})$ leaves $Q_1$ if there exists a positive value of $t$ for which the first coordinate is 0. Equivalently, the solution leaves $Q_1$ if there exists $c>0$ such that $(1-\lambda)e^{\lambda}=(1+c)e^{-c}$, where the time to escape is given by $t=\lambda+c$. From \eqref{eq:f} it follows that $c = C(\lambda)$ which ends the proof.
\end{proof}

In Lemma~\ref{lem:exist_C} we have obtained the point $\mathbf{p}_c=\left(0,\varepsilon(\lambda+C(\lambda))\right)^\top$ whose $y$-coordinate gives an upper bound for the crossing of the $y$-axis. We note that since $C(\lambda)$ tends to $\infty$ as $\lambda \to 1$ and the second coordinate of the vector field is positive, reaching the attracting manifold $\mathcal{S}_{a,\varepsilon}^+$ at an ordinate $\rho$ introduces naturally a maximal value for $\varepsilon$ as a function of $\lambda$. Consider the function $\varepsilon_0:(0,1)\to \mathbb{R}^+$ given by
\begin{equation}\label{def:eps0}
\varepsilon_0(\lambda)=\frac {e^{1-C(\lambda)}}{C(\lambda)-1},
\end{equation}
where $C(\lambda)$ is the analytical function implicitly given in \eqref{eq:f}. Straightforward computations show that 
\[
\lim_{\lambda \to 1} \varepsilon_0(\lambda)=0, \quad \lim_{\lambda \to 1} \varepsilon_0'(\lambda)=-\infty. 
\]
In the next lemma we study the evolution of the solution through $\mathbf{p}_c$ after entering quadrant $Q_4$. 

\begin{lemma}\label{lem:exist_eps0}
Fix $\lambda \in (\ln(2),1)$ and let $\varphi(t;\mathbf{p}_c)$ be the solution with initial condition at $\mathbf{p}_c=\left(0,\varepsilon(\lambda+C(\lambda))\right)^\top$. 
\begin{itemize}
\item [a)] For $t\geq 0$, the solution $\varphi(t;\mathbf{p}_c)$ is contained in the quadrant $Q_4$.
\item [b)] Consider $0<\varepsilon\leq \varepsilon_0(\lambda)$, for $\tau_c=\frac {\rho}{\varepsilon}+(C(\lambda)-1)(e^{-\frac {\rho}{\varepsilon}}-1)$, then $d(\varphi(\tau_c,\mathbf{p}_c),\mathcal{S}_{a,\varepsilon}^+ \cap \dout_a)=O(e^{-\frac {\rho}{\varepsilon}})$.
\end{itemize}
\end{lemma}
\begin{proof}
The solution with initial condition $\mathbf{p}_c$ does not intersect with the $x$-nullcline at $t\geq 0$. Hence, it is contained in $Q_4$ for $t\geq 0$ and its local expression is given by
\[
\varphi(t; \mathbf{p}_c)=
\varepsilon
\begin{pmatrix}
(C(\lambda) - 1)(e^{-t}-1) -t \\
(\lambda+C(\lambda))+t
\end{pmatrix}
\]
see equations in \eqref{eq:theFlows-x}-\eqref{eq:theFlows-y}.  Therefore, at time $\tau_c=\frac {\rho}{\varepsilon}+(C(\lambda)-1)(e^{-\frac {\rho}{\varepsilon}}-1)$, the solution satisfies that 
\begin{align*}
\varphi(\tau_c; \mathbf{p}_c)&=\begin{pmatrix}
\varepsilon(C(\lambda) - 1)(e^{-\frac {\rho}{\varepsilon}-(C(\lambda)-1)(e^{-\frac {\rho}{\varepsilon}}-1)}-1) -\rho -\varepsilon(C(\lambda)-1)(e^{-\frac {\rho}{\varepsilon}}-1) \vspace{.2cm} \\
\varepsilon(\lambda+C(\lambda))+\rho+\varepsilon(C(\lambda)-1)(e^{-\frac {\rho}{\varepsilon}}-1)
\end{pmatrix}\\
&=\begin{pmatrix}
-\rho + \varepsilon(C(\lambda) - 1)e^{-\frac {\rho}{\varepsilon}}  \left(e^{-(C(\lambda)-1)(e^{-\frac {\rho}{\varepsilon}}- 1)}-1\right) \vspace{.2cm} \\
\rho+ \varepsilon(\lambda+1)+\varepsilon(C(\lambda)-1)e^{-\frac {\rho}{\varepsilon}}\\
\end{pmatrix}.
\end{align*}
%
Then, 
\[
\varphi(\tau_c; \mathbf{p}_c)-\begin{pmatrix}-\rho\\ \rho+\varepsilon (1+\lambda)\end{pmatrix}=
e^{-\frac {\rho}{\varepsilon}}
\begin{pmatrix}
\varepsilon(C(\lambda) - 1)\left(e^{(C(\lambda)-1)(1-e^{-\frac {\rho}{\varepsilon}})}-1\right) \vspace{.2cm} \\
\varepsilon(C(\lambda) - 1)\\
\end{pmatrix}.
\]
Since $\varepsilon\leq \varepsilon_0(\lambda)$, from \eqref{def:eps0} it follows that 
\[
d\left(\varphi(\tau_c; \mathbf{p}_c)-\begin{pmatrix}-\rho\\ \rho+\varepsilon (1+\lambda)\end{pmatrix}\right)=O(e^{-\frac{\rho}{\varepsilon}}),
\]
which ends the proof. 
\end{proof}

In the next lemma we study the behaviour of the solution $\varphi(t;\mathbf{p}_a^-)$ for $\lambda\in (\ln(2),1)$ after re-entering quadrant $Q_4$. We take profit of the solutions through $\mathbf{p}_t$ and $\mathbf{p}_c$ described in Lemmas~\ref{lm:asm_behav_a} and \ref{lem:exist_eps0}.

\begin{lemma}\label{lm:asm_behav_b} 
Fix $\lambda\in (\ln(2),1)$ and set $0<\varepsilon\leq \varepsilon_0(\lambda)$. Let $\varphi(t;\mathbf{p}_a^-)$ be the solution of System \eqref{eq:transBif} with initial condition at $\mathbf{p}_a^-$. For $\tau=\tau_a+O(e^{-\frac {\rho}{\varepsilon}})$ with $\tau_a$ in Lemma~\ref{lm:asm_behav_a}, then  
$d(\varphi(\tau,\mathbf{p}_a^-),\mathcal{S}_{a,\varepsilon}^+ \cap \dout_a)=O(e^{-\frac {\rho}{\varepsilon}})$.
\end{lemma}
\begin{proof}
Notice that $\varphi(t;\mathbf{p}_a^-)$ is located between $\varphi(t-\lambda;\mathbf{p}_t)$ and $\varphi(t-(\lambda+C(\lambda));\mathbf{p}_c)$ for $t>0$. Therefore, the proof follows straightforwardly from Lemmas~\ref{lm:asm_behav_a}(b) and \ref{lem:exist_eps0}(b).
\end{proof}

Next, we start the proof of the statement (b) of Theorem~\ref{th:main1}. To that aim, let us first consider the point $\pin_\delta = \big(-\rho, -\rho + \varepsilon(1-\lambda) + \delta \big) \in \din$. In order to reach $\dout$, we first need to integrate it until $y=0$ (the change from quadrant $Q_3$ to quadrant $Q_4$). One can easily show that for $\tau=\frac{\rho+\delta}{\varepsilon}+1-\lambda$ we get
\begin{equation}
\varphi(\tau, \pin_\delta) = \mathbf{p}_a^-+ \delta e^{-\frac {\rho+\delta}{\varepsilon}+(1-\lambda)}\mathbf{e}_1
\end{equation}
which is exponentially close to $\mathbf{p}_a^-$.

For $\lambda \leq \ln(2)$, following similar arguments than those in the proof of Lemma~\ref{lm:asm_behav_a} we conclude that for $\tau=\frac{\rho + \delta}{\varepsilon}+ 1+\lambda$ then
\[
d(\varphi(\tau,\mathbf{p}_a^-+ \delta e^{-\frac {\rho+\delta}{\varepsilon}+(1-\lambda)}\mathbf{e}_1),\mathcal{S}_{a,\varepsilon}^+ \cap \dout_a)=O(e^{-\frac {\rho}{\varepsilon}}).
\]
On the contrary, for $\ln(2)<\lambda< 1$, similarly as in the proof of Lemma~\ref{lm:asm_behav_b}, the solution $\varphi(t;\mathbf{p}_a^-+ \delta e^{-\frac {\rho+\delta}{\varepsilon}+(1-\lambda)}\mathbf{e}_1)$ is located between the solutions $\varphi(t-\lambda;\mathbf{p}_t)$ and $\varphi(t-(\lambda+C(\lambda));\mathbf{p}_c)$ for $t>0$. Hence, 
\[
d(\varphi(\tau,\mathbf{p}_a^-+ \delta e^{-\frac {\rho+\delta}{\varepsilon}+(1-\lambda)}\mathbf{e}_1),\mathcal{S}_{a,\varepsilon}^+ \cap \dout_a)=O(e^{-\frac {\rho}{\varepsilon}}),
\]
with $\tau=\tau_a+O(e^{-\frac{\rho}{\varepsilon}})$, which ends the proof of the statement (b) in Theorem~\ref{th:main1}.

\subsubsection{Case $\lambda > 1$}

Consider now the function $\varepsilon_1:(1,+\infty) \to \mathbb{R}$ given by
\begin{equation}\label{def:ep1}
    \varepsilon_1(\lambda)= \min \left\{  \left(\frac {2(e^{\lambda-1}-1)}{\rho+\lambda-1} \right)^{\frac 1 2},
    \frac {\rho}{2e^{\lambda-1}-(\lambda+1)}
     \right\}.    
\end{equation}
Notice that for $\lambda$ close to 1 the value of $\varepsilon_1(\lambda)$ is given by the first function in \eqref{def:ep1}. As a consequence, $\lim_{\lambda \to 1 } \varepsilon_1(\lambda)=0$ and $\lim_{\lambda\to 1} \frac {d\varepsilon_1(\lambda)}{d\lambda}=-\infty.$

When $\lambda>1$, the slow manifold $\mathcal{S}_{a,\varepsilon}^-$ leaves the region $Q_3$ at the point $\mathbf{p}_a^-=(0,\varepsilon(1-\lambda))$, see \eqref{eq:pa-}. Hence, by following the flow in the region $Q_2$, the solution through $\mathbf{p}_a^-$ writes as 
\[
\varphi(t;\mathbf{p}_a^-)=\begin{pmatrix} 2\varepsilon (e^t-1) -\varepsilon t \\ (1-\lambda)\varepsilon +\varepsilon t\end{pmatrix},
\]
see \eqref{eq:theFlows-x}, for $t\in(0, \tau_1)$ where $\tau_1:=\lambda-1$. Moreover, the solution $\varphi(t; \mathbf{p}_a^-)$ crosses from the quadrant $Q_2$ to the quadrant $Q_1$ through the point $\mathbf{p}_1=\varphi(\tau_1;\mathbf{p}_a^-)$. In the next result we describe the behaviour of the solution through $\mathbf{p}_1$.  

\begin{lemma}\label{lem:exist_eps1}
Fix $\lambda >1 $ and set $0<\varepsilon\leq \varepsilon_1(\lambda)$. Let $\varphi(t;\mathbf{p}_1)$ be the solution with initial condition at the point $\mathbf{p}_1=(2\varepsilon e^{\lambda-1}-\varepsilon(1+\lambda),0)^\top$. 
\begin{itemize}
\item [a)] For $t\geq 0$, the solution $\varphi(t;\mathbf{p}_1)$ is contained in the quadrant $Q_1$.
\item [b)] There exist functions $\tau(\varepsilon)$ and $h(\varepsilon)$ such that  $\varphi(\tau(\varepsilon);\mathbf{p}_1)=(\rho, h(\varepsilon))^\top$, where $\tau(\varepsilon)=O(|\ln(\varepsilon)|)$ and $h(\varepsilon)=O(\varepsilon|\ln(\varepsilon)|)$.
\end{itemize}
\end{lemma}
\begin{proof}
Notice that in the region of $Q_1$ limited by the coordinate axes and by the repelling slow manifold $\mathcal{S}_{r,\varepsilon}^+$, the vector field satisfies $\dot{x}>0$ and $\dot{y}>0$, see \eqref{eq:transBif}. Therefore, both coordinates of the solution $\varphi(t;\mathbf{p}_1)$ are strictly increasing and hence, it is contained in $Q_1$ for $t\geq 0$.

Let $x_1=\varepsilon(2e^{\lambda-1}-(\lambda+1))>0$ be the first coordinate of the point $\mathbf{p}_1$. Since 
\[
\varepsilon < \varepsilon_1(\lambda)\leq \frac {\rho}{2e^{\lambda-1}-(\lambda+1)},
\] 
it follows that  $x_1 < \rho$. As both coordinates of the solutions are increasing with the time, there exist positive functions $\tau(\varepsilon)$ and $h(\varepsilon)$ such that $\varphi(\tau(\varepsilon);\mathbf{p}_1)=(\rho,h(\varepsilon))^\top$. Next, we study the order in $\varepsilon$ of the functions $\tau(\varepsilon)$ and $h(\varepsilon
)$.

Set $d=-\ln(\varepsilon^3)>0$ and consider the first coordinate 
$
x_d= (\rho-\varepsilon d +(\lambda-1)\varepsilon)e^{-d}-(\lambda-1)\varepsilon,
$ 
of the point $\mathbf{p}_d$ introduced in Lemma \ref{lem:tech}. Hence, 
\begin{align*}
    x_d &=(\rho+ 3\varepsilon \ln(\varepsilon) +(\lambda-1)\varepsilon)\varepsilon^3-(\lambda-1)\varepsilon\\
        &<\varepsilon \left( (\rho + \lambda-1 )\varepsilon^2-\lambda+1\right) \\
        &\leq \varepsilon \left( (\rho + \lambda-1 )\varepsilon_1^2(\lambda)-\lambda+1\right) \\
        &=\varepsilon (2e^{\lambda-1}-(\lambda+1)),
\end{align*}
that is $x_d<x_1$. Consequently, it follows that the orbit defined by the solution  $\phi(t;\mathbf{p}_d)$ introduced in Lemma \ref{lem:tech} bounds in $Q_1$ the orbit defined by $\varphi(t;\mathbf{p}_1)$, including their intersections with $\Delta_{e}^{out}$. From Lemma \ref{lem:tech}, we conclude that $h(\varepsilon)=O(\varepsilon|\ln(\varepsilon)|)$ and $\tau(\varepsilon)=O(|\ln(\varepsilon)|)$, which ends the proof. 
\end{proof}

Lemma \ref{lem:exist_eps1} proves the first part of the statement (a) of Theorem \ref{th:main1}. Next, we finish this proof. To that aim, let us consider the point $\pin_\delta = \big(-\rho, -\rho + \varepsilon(1-\lambda) + \delta \big) \in \din$. In order to reach $\dout$, we first need to integrate it until $x=0$ (the change from quadrant $Q_3$ to quadrant $Q_2$). One can easily show that for $\widetilde{\tau}_1=\frac{\rho-\delta}{\varepsilon}$ we get
\begin{equation}
\varphi(\widetilde{\tau}_1; \pin_\delta) = \left( 0, \varepsilon(1-\lambda) - \delta e^{-\frac {\rho-\delta}{\varepsilon}}\right)^\top=\mathbf{p}_a^- - \delta e^{-\frac {\rho-\delta}{\varepsilon}}\mathbf{e}_2,  
\end{equation}
which is 
exponentially close to $\mathbf{p}_a^-$. 

Next, we compute the solution with initial condition at the point $\varphi(\widetilde{\tau}_1; \pin_\delta)$  until it reaches the $y$-axis (where it changes from quadrant $Q_2$ to quadrant $Q_1$). Straightforward computations show that for $\widetilde{\tau}_2=\lambda - 1 + \frac{\delta}{\varepsilon}e^{-\widetilde{\tau}_1}$ we get
\begin{equation}
\widetilde{\mathbf{p}}_1:=\varphi(\widetilde{\tau}_2; \varphi(\widetilde{\tau}_1; \pin_\delta) ) = \left(\left( 2\varepsilon   - \delta e^{-\widetilde{\tau}_1} \right) e^{\widetilde{\tau}_2} - \varepsilon  (\lambda+1),0\right)^\top.    
\end{equation}
Since $\frac {\delta}{\varepsilon}e^{-\widetilde{\tau}_1}$ tends to zero with  $\varepsilon$, we write 
\[
e^{\widetilde{\tau}_2}=e^{\lambda-1}\left(1+\frac{\delta}{\varepsilon} e^{-\widetilde{\tau}_1}+O\left( e^{-2\widetilde{\tau}_1}\right)\right)
=e^{\lambda-1}+\frac {\delta}{\varepsilon}e^{\lambda-1-\widetilde{\tau}_1}
+O\left(e^{-2\widetilde{\tau}_1}\right),
\]
and consequently
\[
\widetilde{\mathbf{p}}_1=\mathbf{p}_1+\left(\delta e^{\lambda-1-\widetilde{\tau}_1} +O\left( e^{-2\widetilde{\tau}_1}\right),0\right)^\top,
\]
where $\mathbf{p}_1$ is introduced in Lemma \ref{lem:exist_eps1}. We conclude that $\widetilde{\mathbf{p}}_1$ and $\mathbf{p}_1$ are exponentially close and $\varphi(t;\widetilde{\mathbf{p}}_1)\subset Q_1$ for $t>0$. Therefore 
\[
\varphi(t;\widetilde{\mathbf{p}}_1)=\varphi(t;{\mathbf{p}}_1)+ \left( \left(\delta e^{\lambda-1-\widetilde{\tau}_1} +O\left( e^{-2\widetilde{\tau}_1}\right)\right)e^t,0\right)^\top,
\]
for $t>0$. In particular, for $t=\tau(\varepsilon)$ it follows that 
\[
\varphi(\tau(\varepsilon),\widetilde{\mathbf{p}}_1)=\begin{pmatrix} \rho \\ h(\varepsilon)\end{pmatrix} +
\begin{pmatrix} O\left( \frac {\delta}{\varepsilon} e^{-\widetilde{\tau}_1} \right)\\ 0 \end{pmatrix}. 
\]
This ends the proof of the  statement (a) of Theorem \ref{th:main1}.

\subsection{Proof of Theorem \ref{th:main1delay}}

First, we note that the boundary $O(\varepsilon\ln(|\varepsilon|))$ in the statement (a) follows from Lemma~\ref{lem:exist_eps1}. In addition, the boundary $\varepsilon\,(\lambda+C(\lambda))$ given in the statement (b), follows from Lemmas \ref{lm:asm_behav_a}, \ref{lem:exist_eps0} and \ref{lem:exist_C}. Before proving the boundary $\varepsilon\,(2-\lambda+C(2-\lambda))$ in the statement (a), we present the following lemma:

\begin{lemma}\label{lem:exist_C2}
Consider $\lambda\in (1,2)$. The solution of System \eqref{eq:transBif} with initial condition at the origin, $\varphi(t;\mathbf{0})$, is locally contained in $Q_1$ for $t >0 $ and  
\[
d(\mathcal{S}_{r,\varepsilon}^+,\varphi(2-\lambda+C(2-\lambda);\mathbf{0}))= d(\mathcal{S}_{r,\varepsilon}^+,\varphi(0;\mathbf{0}))e^{2-\lambda+C(2-\lambda)},
\] 
where the function $C(\lambda)$ is defined in \eqref{eq:f}.
\end{lemma}
\begin{proof}
From the expression of the flow in the quadrant $Q_1$ (see \eqref{eq:theFlows-x}-\eqref{eq:theFlows-y}) it follows that $\varphi(t;\mathbf{0})=(x(t),y(t))^\top$ satisfies $y(t)=x(t)-\varepsilon(1-\lambda)-\varepsilon(\lambda-1)e^t$, for $t>0$. The proof follows by taking into account the expression of $\mathcal{S}_{r,\varepsilon}^+$ in \eqref{def:rpslm}.
\end{proof}

Assuming $1<\lambda<2-\ln(2)$, from Lemma \ref{lem:exist_C2}, the distance of the slow manifold $\mathcal{S}_{r,\varepsilon}^+$ to the solution through the origin at time $\tau_0=2-\lambda+C(2-\lambda)$ satisfies that 
\[
d(\mathcal{S}_{r,\varepsilon}^+,\varphi(2-\lambda+C(2-\lambda);\mathbf{0}))> 2 e^{C(\ln(2))}d(\mathcal{S}_{r,\varepsilon}^+,\varphi(0;\mathbf{0}))>2 d(\mathcal{S}_{r,\varepsilon}^+,\varphi(0;\mathbf{0})).
\]
Consequently, the delay in the loss of stability for the solution through the origin is smaller that $\tau_0$. Moreover, from Lemma \ref{lem:exist_eps1} the attracting slow manifold $\mathcal{S}_{a,\varepsilon}^-$ evolves up to the quadrant $Q_1$ through the point $\mathbf{p}_1=(2\varepsilon e^{\lambda-1}-\varepsilon(1+\lambda),0)^\top$. Since the first coordinate of $\mathbf{p}_1$ is positive, the orbit through the origin bounds the orbit through $\mathbf{p}_1$. From this, we conclude that the delay in the loss of stability of $\mathcal{S}_{a,\varepsilon}^-$ is smaller that $\tau_0$. This implies that the maximal delay is smaller than $\varepsilon \tau_0$, which ends the proof of  Theorem \ref{th:main1delay}.

\subsection{Proof of Theorem \ref{th:main2}}

When $\lambda=1$, both slow manifolds do connect at the origin as it follows from the expression of $\mathcal{S}_{a,\varepsilon}^-$ and $\mathcal{S}_{r,\varepsilon}^+$ given in \eqref{def:atrslm} and \eqref{def:rpslm}, respectively. Moreover, the slow manifolds lie on the straight line $y=x$, which is invariant under the flow. Therefore, the orbit 
through $\mathbf{p}_a^-=(0,0)$ remains on the invariant set and thus the maximal delay is unbounded.

However, we may now discuss what happens with the points inside a neighbourhood $N_\delta$ around the slow-manifold. To do that, let us consider an orbit of System \eqref{eq:transBif} entering the neighbourhood $N_{\delta}$ at a point $\mathbf{p}_i=(-\rho,-\rho+\delta)$. By following the flow we are interested in the point $\mathbf{p}_o$ at which it leaves $N_{\delta}$. Notice that, for $\tau_1=\frac {\rho-r}{\varepsilon}$, then  $\mathbf{p}_1=\varphi(\tau_1;\mathbf{p}_{i})=(-\delta e^{-\frac {\rho-\delta}{\varepsilon}},0)^T$, see the expression of the flow in $Q_3$ given in~\eqref{eq:theFlows-x}-\eqref{eq:theFlows-y}.  Moreover, there exists a time of flight 
$$\tau_2=\frac{\delta}{\varepsilon} e^{-\frac {\rho-\delta}{\varepsilon}}+O\left( e^{-2\frac {\rho-\delta}{\varepsilon}} \right),$$
such that 
$$\mathbf{p}_2=\varphi(\tau_2;\mathbf{p}_1)=\left(0,\delta e^{-\frac{\rho-\delta}{\varepsilon}}+O\left( e^{-2\frac {\rho-\delta}{\varepsilon}}\right)\right).
$$
Finally, for the times of flight 
$$\tau_3= \frac {\rho-\delta}{\varepsilon}-\frac {\delta}{\varepsilon}e^{-\frac{\rho-\delta}{\varepsilon}},\text{~~and~~}
\tau_4=\frac {\rho}{\varepsilon}-\frac {\delta}{\varepsilon}e^{-\frac{\rho-\delta}{\varepsilon}},
$$
we obtain 
$$\varphi(\tau_3;\mathbf{p}_2)=
\begin{pmatrix}
\rho-\delta-\delta e^{-\frac {\delta}{\varepsilon} e^{-\frac{\rho-\delta}{\varepsilon}}}+O(e^{-\frac{\rho-\delta}{\varepsilon}})\\
\rho-\delta+O(e^{-2\frac{\rho-\delta}{\varepsilon}})
\end{pmatrix}, \text{~and~}
\varphi(\tau_4;\mathbf{p}_2)=
\begin{pmatrix}
\rho-\delta e^{\frac {\delta}{\varepsilon} (1-e^{-\frac{\rho-\delta}{\varepsilon}})}+O(e^{-\frac{\rho-\delta}{\varepsilon}})\\
\rho+O(e^{-2\frac{\rho-\delta}{\varepsilon}})
\end{pmatrix}.
$$ 
Therefore, we conclude that the exit point $\mathbf{p}_o$ has a second coordinate between $\rho-\delta$ and $\rho$. The same arguments apply if we consider $\mathbf{p}_i=(-\rho,-\rho-\delta)^T$. This ends the proof of the theorem. 

\subsection{Proof of Theorem \ref{th:main3}}

In this section we prove Theorem~\ref{th:main3}. Let us start with the statement (a). Assuming $c>0$ and $\varepsilon>0$ and small enough, the parameter $\lambda=1-e^{-\frac{c}{\varepsilon}}$ is greater than $\ln(2)$ and therefore the solution through $\mathbf{p}_a^-$ evolves from quadrant $Q_4$ to $Q_1$ intersecting the $y$-axis at a point $\mathbf{p}_1$ below the tangent point $\mathbf{p}_t$. In particular, for the time of flight $\tau_1=e^{-\frac {c}{\varepsilon}}+O(e^{-\frac {2c}{\varepsilon}})$ such that $\mathbf{p}_1=\varphi(\tau_1;\mathbf{p}_a^-)=(0,y_1)^T$ with $y_1=\varepsilon e^{-\frac {c}{\varepsilon}}+O(\varepsilon e^{-\frac {2c}{\varepsilon}})$, see Equations \eqref{eq:theFlows-x}-\eqref{eq:theFlows-y}. Moreover, for the time of flight $\tau_2=\frac {c}{\varepsilon}+\ln(\frac{c}{2\varepsilon})$ such that $\varphi(\tau_2;\mathbf{p}_1)=(0+X(\varepsilon),c+Y(\varepsilon))^T$ with 
\[
X(\varepsilon)= \varepsilon \ln\left(\frac c{2\varepsilon}\right) + O(\varepsilon e^{-\frac {c}{\varepsilon}}), \text{~~and,~~}
Y(\varepsilon)= \varepsilon \ln\left(\frac c{2\varepsilon}\right) + O(e^{-\frac {c}{\varepsilon}}).
\]

Let us now prove the statement (b). The parameter $\lambda=1+e^{-\frac{c}{\varepsilon}}$ is greater than $1$ and therefore the solution through $\mathbf{p}_a^-$ evolves from quadrant $Q_3$ to $Q_1$ intersecting the $x$-axis at a point $\mathbf{p}_1$. In particular, there exists a time of flight $\tau_1=e^{-c/\varepsilon}$ such that $\mathbf{p}_1=\varphi(\tau_1;\mathbf{p}_a^-)=(x_1,0)^T$ with $x_1=\varepsilon e^{-c/\varepsilon}+O(e^{-2c/\varepsilon})$, see Equations \eqref{eq:theFlows-x}-\eqref{eq:theFlows-y}. Moreover, for $\tau_2=\frac {c}{\varepsilon}+\ln(\frac{c}{2\varepsilon})$ it can be checked that $\varphi(\tau_2;\mathbf{p}_1)=(2c + X(\varepsilon),c+Y(\varepsilon))^T$ with 
\[
X(\varepsilon)=  \varepsilon \ln\left(\frac c{2\varepsilon}\right) + O(\varepsilon e^{-\frac {c}{\varepsilon}}), \text{~~and,~~} Y(\varepsilon)=  \varepsilon \ln\left(\frac c{2\varepsilon}\right).
\]

\section{Acknowledgements}
AET is partially supported by the MCIU project
PID2020-118726GB-I00 and by the Ministerio de Economia y Competitividad through the project MTM2017-83568-P (AEI/ERDF,EU). APC thanks the Departament de Matemàtiques i Informàtica of the UIB and the Instituto de Fisica Interdisciplinar y Sistemas Complejos (IFISC, UIB-CSIC) for hosting him during pandemic times. APC acknowledges support from Spanish Ministry of Science and Innovation grants (Projects No. PID2021-124047NB-I00 and PID-2021-122954NB-100). 

\bibliographystyle{abbrv}  
\bibliography{refs}

\begin{thebibliography}{10}

\bibitem{BER89}
S.~M. Baer, T.~Erneux, and J.~Rinzel.
\newblock The slow passage through a hopf bifurcation: Delay, memory effects,
  and resonance.
\newblock {\em SIAM Journal on Applied Mathematics}, 49(1):55--71, 1989.

\bibitem{B81}
E.~Benoit.
\newblock Chasse au canard.
\newblock {\em Collectanea Mathematica}, 32(2):37--119, 1981.

\bibitem{CFT23}
V.~Carmona, S.~Fern{\'{a}}ndez-Garc{\'{i}}a, and A.~E. Teruel.
\newblock {Birth, transition and maturation of canard cycles in a piecewise
  linear system with a flat slow manifold}.
\newblock {\em Physica D: Nonlinear Phenomena}, 443:133566, 2023.

\bibitem{carmona2008existence}
V.~Carmona, F.~Fern{\'a}ndez-S{\'a}nchez, and A.~E. Teruel.
\newblock Existence of a reversible t-point heteroclinic cycle in a piecewise
  linear version of the michelson system.
\newblock {\em SIAM journal on applied dynamical systems}, 7(3):1032--1048,
  2008.

\bibitem{CFT20}
V.~Carmona, S.~Fernández-García, and A.~Teruel.
\newblock Saddle–node canard cycles in slow–fast planar piecewise linear
  differential systems.
\newblock {\em Nonlinear Analysis: Hybrid Systems}, 52:101472, 2024.

\bibitem{DDR21}
P.~de~Maesschalck, F.~Dumortier, and R.~Roussarie.
\newblock {\em {Canard Cycles: from Birth to Transition}}, volume~73 of {\em
  Ergebnisse der Mathematik und ihrer Grenzgebiete. 3. Folge / A Series of
  Modern Surveys in Mathematics}.
\newblock {Springer International Publishing}, 2021.

\bibitem{DFKPT18}
M.~Desroches, S.~Fern{\'a}ndez-Garc{\'i}a, M.~Krupa, R.~Prohens, and A.~E.
  Teruel.
\newblock Piecewise-linear (pwl) canard dynamics - simplifying singular
  perturbation theory in the canard regime using piecewise-linear systems.
\newblock In V.~Carmona, J.~Cuevas-Maraver, F.~Fern{\'a}ndez-S{\'a}nchez, and
  E.~Garc{\'i}a-Medina, editors, {\em Nonlinear Systems, Vol. 1: Mathematical
  Theory and Computational Methods}, pages 67--86. Springer International
  Publishing, Cham, 2018.

\bibitem{DGPPRT16}
M.~Desroches, A.~Guillamon, E.~Ponce, R.~Prohens, S.~Rodrigues, and A.~E.
  Teruel.
\newblock Canards, folded nodes, and mixed-mode oscillations in
  piecewise-linear slow-fast systems.
\newblock {\em SIAM Review}, 58(4):653--691, 2016.

\bibitem{D94}
M.~Diener.
\newblock Regularizing microscopes and rivers.
\newblock {\em SIAM Journal on Mathematical Analysis}, 25(1):148--173, 1994.

\bibitem{EK19}
M.~Engel and C.~Kuehn.
\newblock {Discretized fast-slow systems near transcritical singularities}.
\newblock {\em Nonlinearity}, 32(7):2365, may 2019.

\bibitem{ERHG91}
T.~Erneux, E.~L. Reiss, L.~J. Holden, and M.~Georgiou.
\newblock Slow passage through bifurcation and limit points. asymptotic theory
  and applications.
\newblock In E.~Beno{\^i}t, editor, {\em Dynamic Bifurcations}, pages 14--28,
  Berlin, Heidelberg, 1991. Springer Berlin Heidelberg.

\bibitem{F79}
N.~Fenichel.
\newblock Geometric singular perturbation theory for ordinary differential
  equations.
\newblock {\em Journal of Differential Equations}, 31(1):53--98, 1979.

\bibitem{FDKT15}
S.~Fernández-García, M.~Desroches, M.~Krupa, and A.~Teruel.
\newblock Canard solutions in planar piecewise linear systems with three zones.
\newblock {\em Dynamical Systems An International Journal}, 31:173--197, 08
  2016.

\bibitem{F13}
A.~F. Filippov.
\newblock {\em Differential equations with discontinuous righthand sides:
  control systems}, volume~18.
\newblock Springer Science \& Business Media, 2013.

\bibitem{FPV08}
J.-P. Françoise, C.~Piquet, and A.~Vidal.
\newblock Enhanced delay to bifurcation.
\newblock {\em Bulletin of the Belgian Mathematical Society-Simon Stevin},
  15(5):825--831, 2008.

\bibitem{GST11}
M.~Guardia, T.~Seara, and M.~A. Teixeira.
\newblock Generic bifurcations of low codimension of planar filippov systems.
\newblock {\em Journal of Differential equations}, 250(4):1967--2023, 2011.

\bibitem{H79}
R.~Haberman.
\newblock Slowly varying jump and transition phenomena associated with
  algebraic bifurcation problems.
\newblock {\em SIAM Journal on Applied Mathematics}, 37(1):69--106, 1979.

\bibitem{H00}
R.~Haberman.
\newblock Slow passage through a transcritical bifurcation for hamiltonian
  systems and the change in action due to a nonhyperbolic homoclinic orbit.
\newblock {\em Chaos: An Interdisciplinary Journal of Nonlinear Science},
  10(3):641--648, 2000.

\bibitem{HKSW16}
M.~G. Hayes, T.~J. Kaper, P.~Szmolyan, and M.~Wechselberger.
\newblock {Geometric desingularization of degenerate singularities in the
  presence of fast rotation: A new proof of known results for slow passage
  through Hopf bifurcations}.
\newblock {\em Indagationes Mathematicae}, 27(5):1184--1203, 2016.

\bibitem{J95}
C.~K. R.~T. Jones.
\newblock {\em Geometric singular perturbation theory}, pages 44--118.
\newblock Springer Berlin Heidelberg, Berlin, Heidelberg, 1995.

\bibitem{K23}
P.~Kaklamanos, C.~Kuehn, N.~Popovi{\'c}, and M.~Sensi.
\newblock Entry--exit functions in fast--slow systems with intersecting
  eigenvalues.
\newblock {\em Journal of Dynamics and Differential Equations}, pages 1--18,
  2023.

\bibitem{krupa2001extending}
M.~Krupa and P.~Szmolyan.
\newblock Extending slow manifolds near transcritical and pitchfork
  singularities.
\newblock {\em Nonlinearity}, 14(6):1473, 2001.

\bibitem{L75}
N.~Lebovitz and R.~Schaar.
\newblock Exchange of stabilities in autonomous systems.
\newblock {\em Studies in Applied Mathematics}, 54(3):229--260, 1975.

\bibitem{L77}
N.~Lebovitz and R.~Schaar.
\newblock Exchange of stabilities in autonomous systems—ii. vertical
  bifurcation.
\newblock {\em Studies in Applied Mathematics}, 56(1):1--50, 1977.

\bibitem{llibre2007horseshoes}
J.~Llibre, E.~Ponce, and A.~E. Teruel.
\newblock Horseshoes near homoclinic orbits for piecewise linear differential
  systems in $\mathbf{R}^3$.
\newblock {\em International Journal of Bifurcation and Chaos},
  17(04):1171--1184, 2007.

\bibitem{N87}
A.~Neishtadt.
\newblock Persistence of stability loss for dynamical bifurcations i.
\newblock {\em Differential Equations}, 23:1385--1391, 1987.

\bibitem{N88}
A.~Neishtadt.
\newblock Persistence of stability loss for dynamical bifurcations ii.
\newblock {\em Differential Equations}, 24:171--176, 1988.

\bibitem{PDTV22}
J.~Penalva, M.~Desroches, A.~E. Teruel, and C.~Vich.
\newblock Slow passage through a hopf-like bifurcation in piecewise linear
  systems: Application to elliptic bursting.
\newblock {\em Chaos: An Interdisciplinary Journal of Nonlinear Science},
  32(12):123109, 2022.

\bibitem{PDTV23}
J.~Penalva, M.~Desroches, A.~E. Teruel, and C.~Vich.
\newblock Dynamics of a piecewise-linear morris-lecar model: bifurcations and
  spike adding.
\newblock {\em arXiv}, 2023.

\bibitem{ponce2022bifurcations}
E.~Ponce, J.~Ros, and E.~Vela.
\newblock {\em Bifurcations in Continuous Piecewise Linear Differential
  Systems: Applications to Low-Dimensional Electronic Oscillators}, volume~7.
\newblock Springer Nature, 2022.

\bibitem{PTV16}
R.~Prohens, A.~Teruel, and C.~Vich.
\newblock {Slow–fast n-dimensional piecewise linear differential systems}.
\newblock {\em Journal of Differential Equations}, 260(2):1865--1892, jan 2016.

\bibitem{FV12}
A.~Vidal and J.-P. Fran{\c{c}}oise.
\newblock Canard cycles in global dynamics.
\newblock {\em International Journal of Bifurcation and Chaos}, 22(02):1250026,
  2012.

\end{thebibliography}





\newpage
\appendix
\section{Local expressions for the flow}\label{sec:appendix}

We provide local expressions for the flow $\varphi(t;(x_0,y_0),\lambda,\varepsilon)$ of the PWL  minimal model in \eqref{eq:transBif} depending on $(x_0,y_0)$ is located in  each of the four quadrants of the phase-space, which we denote by $Q_i$ (with $i=1,\dots,4$) following the clockwise convention. 
For simplify notation, when no confusion arises we will avoid make explicit the dependence of the flow with respect to the parameters $\lambda$ and $\varepsilon$.
In particular, the first coordinate of the flow is given by 
\begin{equation}\label{eq:theFlows-x}
\varphi_1(t;(x_0,y_0))=\left\{
\begin{array}{ll}
 (x_0 - y_0 + (\lambda-1)\varepsilon)e^{t} + \varepsilon t + y_0 -  (\lambda-1)\varepsilon &
 (x_0,y_0)\in Q_1,\\
 (x_0 + y_0 + (1+\lambda)\varepsilon)e^{t} -\varepsilon t - y_0 - (\lambda+1)\varepsilon &
 (x_0,y_0)\in Q_2,\\
 (x_0 - y_0 + (1-\lambda)\varepsilon)e^{-t} + \varepsilon t + y_0 +  (\lambda-1)\varepsilon &
 (x_0,y_0)\in Q_3,\\ 
 (x_0 + y_0 - (1+\lambda)\varepsilon)e^{-t} -\varepsilon t - y_0 + (\lambda+1)\varepsilon &
 (x_0,y_0)\in Q_4, 
\end{array}
\right.
\end{equation}
whereas the second coordinate is given by 
\begin{equation}\label{eq:theFlows-y}
\varphi_2(t;(x_0,y_0))=y_0+\varepsilon t.
\end{equation}

Here we introduce the following technical result.

\begin{lemma}\label{lem:tech}
Consider the linear differential system $\dot{x}=x-y+\varepsilon \lambda,\ \dot{y}=\varepsilon$ defined in whole $\mathbb{R}^2$, and let $\phi(t;\mathbf{p})$ be the solution with initial condition at the point $\mathbf{p}_d=\left( (\rho-\varepsilon d+(\lambda-1)\varepsilon)e^{-d}-(\lambda-1)\varepsilon,0\right)^\top$ with $d>0$. Then, $\phi(d; \mathbf{p}_d)= (\rho,\varepsilon d)^\top$. 
\end{lemma}
\begin{proof}
The proof follows by direct computations from the expression of the flow in $Q_1$ given in  $\eqref{eq:theFlows-x}-\eqref{eq:theFlows-y}$.
\end{proof}



\end{document}